\documentclass[10pt]{amsart}

\usepackage{comment, amsmath, amssymb, amsgen, amsthm, amscd, xspace, color, epsfig, float, epic, multirow, array, mathbbol}
\usepackage{extarrows}

\makeatletter
\newcommand{\rmnum}[1]{\romannumeral #1}
\newcommand{\Rmnum}[1]{\expandafter\@slowromancap\romannumeral #1@}

\newcommand{\lam}{\lambda}

\newcommand{\Ext}{\operatorname{Ext}}

\newcommand{\End}{\operatorname{End}}

\newcommand{\Hom}{\operatorname{Hom}}

\newcommand{\rank}{\mathrm{rank}}

\newcommand{\Mod}{\mathrm{mod}}

\newcommand{\Soc}{\mathrm{Soc}}
\newcommand{\Rad}{\mathrm{Rad}}
\newcommand{\GK}{\mathrm{GKdim}}

\newcommand{\ra}{\rightarrow}

\newcommand{\frb}{\mathfrak{b}}

\newcommand{\frg}{\mathfrak{g}}
\newcommand{\frh}{\mathfrak{h}}

\newcommand{\frl}{\mathfrak{l}}

\newcommand{\frn}{\mathfrak{n}}

\newcommand{\frp}{\mathfrak{p}}

\newcommand{\fru}{\mathfrak{u}}

\newcommand{\frsl}{\mathfrak{sl}}

\newcommand{\bbC}{\mathbb{C}}

\newcommand{\bbZ}{\mathbb{Z}}

\newcommand{\caO}{\mathcal{O}}

\newtheorem{theorem}[equation]{Theorem}
\newtheorem{cor}[equation]{Corollary}
\newtheorem{prop}[equation]{Proposition}
\newtheorem{lemma}[equation]{Lemma}

\theoremstyle{remark}
\newtheorem{remark}[equation]{Remark}

\theoremstyle{definition}
\newtheorem{definition}[equation]{Definition}

\newtheorem{example}[equation]{Example}

\numberwithin{equation}{section}
\begin{document}

\title[Jantzen coefficients]{Jantzen coefficients and radical filtrations for generalized Verma modules}

\author{Jun Hu}
\address[Hu]{School of Mathematics and Statistics, Beijing Institute of Technology, Beijing, 100081, P.R.~China}
\email{junhu404@bit.edu.cn}

\author{Wei Xiao}
\address[Xiao]{College of Mathematics and statistics, Shenzhen University,
Shenzhen, 518060, Guangdong, P. R. China}
\email{xiaow@szu.edu.cn}

\thanks{The first author is supported by Natural Science Foundation of China (Grant No. 11525102). The second author is supported by the National Science Foundation of China (Grant No. 11701381) and Guangdong Natural Science Foundation (Grant No. 2017A030310138).}

\subjclass[2010]{17B10, 22E47}

\keywords{category $\caO^\frp$, block, basic generalized Verma module, Jantzen's simplicity criterion, decomposable pair}

\keywords{Kazhdan-Lusztig polynomial; Jantzen coefficient; basic generalized Verma module; radical filtration}

\bigskip

\begin{abstract}
In this paper we give a sum formula for the radical filtration of generalized Verma modules in any (possibly singular) blocks of parabolic BGG category which can be viewed as a generalization of Jantzen sum formula for Verma modules in the usual BGG category $\caO$. Combined with Jantzen coefficients, we determine the radical filtrations for all basic generalized Verma modules. The proof makes use of the graded version of parabolic BGG category. Explicit formulae for the graded decomposition numbers and inverse graded decomposition numbers of generalized Verma modules in any (possibly singular) integral blocks of the parabolic BGG category are also given.
\end{abstract}


\maketitle

%
%
\section{Introduction}
%
%


The Jantzen filtration and Jantzen sum formula of Verma module are powerful tool for determining the simplicity of Verma modules and computing the character formulae of simple modules in the usual BGG category $\caO$. Each Verma module in $\caO$ is rigid in the sense that its radical filtration coincides with its socle filtration. The work \cite{BB} of Beilinson and Bernstein shows that the Jantzen filtration of each regular Verma module coincides with its radical filtration (and hence its socle filtration). For parabolic version of BGG category $\caO$ (\cite{R}), the generalized Verma module in singular blocks is in general not rigid anymore. Though Jantzen has also introduced a filtration for each generalized Verma module and developed a determinant formula for the contravariant form in \cite{J1}, it is unclear whether it could play any parallel role as in the usual BGG category $\caO$ case, not to say the relationship between this  ``Jantzen filtration'' and the radical filtration of generalized Verma module in this parabolic setting.

In this paper we prove a sum formula (Theorem \ref{sfthm1}) for the radical filtrations of generalized Verma modules in any (possibly singular) blocks of parabolic BGG category $\caO^{\mathfrak{p}}$. Our starting point is to develop an efficient method to determine radical filtrations of generalized Verma module with singularity. The radical filtrations of generalized Verma modules present critical information of related problems, such as homomorphism between generalized Verma modules \cite{Bo, BC, BEJ, L1, L2, Mat1, Mat2, Mat3, X1, X2} and representation types of blocks \cite{BN, P}. Normally, one can apply the $U_\alpha$-algorithm to compute the radical filtration of a Verma module or generalized Verma module with regular highest weight (\cite{V2, GJ, J3, CC, De, I3, BN}). It depends on some recursive relations of Kazhdan-Lusztig polynomials. This algorithm also works for generalized Verma modules which are not ``too singular'' \cite{BN}. In general, the radical filtration of a singular generalized Verma module can be calculated through a regular one \cite{I2, BN}. However, even when the singular category is small, this process might need information of a very large regular category and trigger a great many recursive computations. The sum formula which we obtained can be viewed as a generalization of Jantzen's sum formula for Verma modules. It gives an evidence that the Jantzen filtration and the radical filtration of generalized Verma module might coincide (Remark \ref{lormk2}). The proof makes essential use of the $\mathbb{Z}$-graded representation of the parabolic BGG category. We present explicit formulae (Theorems \ref{graddecomp11} and \ref{inversegraddecomp11}) for the graded decomposition numbers and inverse graded decomposition numbers of generalized Verma modules in any (possibly singular) integral blocks of the parabolic BGG category. By the Koszul duality of the parabolic BGG category, the grading filtration and radical filtration on each generalized Verma module coincide. Those graded decomposition numbers encode all the information about the radical filtration of generalized Verma modules. Combined with the Jantzen coefficients \cite{XZ} and other results, we can efficiently obtain the radical filtrations of generalized Verma modules with singularity in many cases. As an application, we give radical filtrations of basic generalized Verma modules defined and classified in \cite{XZ}. Those modules are induced from maximal parabolic subalgebras and have maximal nontrivial singularity. The sum formula can also be used to determine the blocks of category $\caO^\frp$ \cite{HXZ, X3} and representation types of the blocks \cite{XZhou}.

The paper is organized as follows. The notations and definitions are presented in Section 2. In Section 3, we use the graded BGG resolution (Lemma \ref{bgg}) to derive the graded decomposition numbers for generalized Verma modules with singularity (Theorem \ref{graddecomp11}), and use the known formulae of the inverse graded decomposition numbers of Verma modules (\cite[Theorem 3.11.4]{BGS}) in the usual BGG category $\caO$ to derive the inverse graded decomposition numbers (Theorem \ref{inversegraddecomp11}) for generalized Verma modules with singularity. The sum formula for the radical filtrations of generalized Verma modules (Theorem \ref{sfthm1}) is proved in Section 4. As an application, we describe the radical filtrations of all the basic generalized Verma in the final section.

%
%
\section{Notations and definitions}
%
%

We mainly adopt the notations in \cite{H3}. Let $\frg$ be a complex semisimple Lie algebra with a fixed Cartan subalgebra $\frh$ contained in a Borel subalgebra $\frb$. Denote by $\Phi\subset\frh^*$ the root system of $(\frg, \frh)$. Let $\Phi^+$ be the positive system corresponding to $\frb$ with a simple system $\Delta\subset\Phi^+$. Every subset $I\subset\Delta$ generates a subsystem $\Phi_I\subset\Phi$ with a positive root system $\Phi^+_I:=\Phi_I\cap\Phi^+$. Denote by $W$ (resp. $W_I$) the Weyl group of $\Phi$ (resp. $\Phi_I$) with longest element $w_0$ (resp. $w_I$). Let $\ell(-)$ be the length function on $W$. The action of $W$ on $\frh^*$ is given by $s_\alpha\lambda=\lambda-\langle\lambda, \alpha^\vee\rangle\alpha$ for $\alpha\in\Phi$ and $\lambda\in\frh^*$. Here $\langle-, -\rangle$ is the bilinear form on $\frh^*$ induced from the Killing form and $\alpha^\vee:=2\alpha/(\alpha, \alpha)$ is the coroot of $\alpha$. We say $\lambda\in\frh^*$ is \emph{regular} if $\langle\lambda, \alpha^\vee\rangle\neq0$ for all roots $\alpha\in\Phi$. Otherwise we say $\lambda$ is \emph{singular}. We say $\lambda$ is \emph{integral} if $\langle\lambda, \alpha^\vee\rangle\in\bbZ$ for all $\alpha\in\Phi$. An integral weight $\lambda\in\frh^*$ is \emph{dominant} (resp. \emph{anti-dominant}) if $\langle\lambda, \alpha^\vee\rangle\in\bbZ^{\geq0}$ (resp. $\langle\lambda, \alpha^\vee\rangle\in\bbZ^{\leq0}$) for all $\alpha\in\Delta$.

Let $\frp_I$ be the standard parabolic subalgebra of $\frg$ corresponding to $I$ with Levi decomposition $\frp_I:=\frl_I\oplus\fru_I$, where $\frl_I$ is the standard Levi subalgebra and $\fru_I$ the nilpotent radical of $\frp_I$. For simplicity, we frequently drop the subscript when $I$ is fixed. Put
\[
\Lambda_I^+:=\{\lambda\in\frh^*\ |\ \langle\lambda, \alpha^\vee\rangle\in\bbZ^{>0}\ \mbox{for all}\ \alpha\in I\}.
\]
For $\lambda\in\Lambda_I^+$, let $F(\lambda-\rho)$ be the finite dimensional simple $\frl_I$-modules of highest weight $\lambda-\rho$, where $\rho:=\frac{1}{2}\sum_{\alpha\in\Phi^+}\alpha$. The \emph{generalized Verma module} is defined by
\[
M_I(\lambda):=U(\frg)\otimes_{U(\frp_I)}F(\lambda-\rho),
\]
where $F(\lambda-\rho)$ has trivial $\fru_I$-actions viewed as a $\frp_I$-module. The highest weight of $M_I(\lambda)$ is $\lambda-\rho$. Let $L(\lambda)$ be the unique simple quotient of $M_I(\lambda)$. They share the same infinitesimal character $\chi_\lambda$, where $\chi_\lambda$ is an algebra homomorphism from the center $Z(\frg)$ of $U(\frg)$ to $\bbC$ so that $z\cdot v=\chi_\lambda(z)v$ for all $z\in Z(\frg)$ and $v\in M_I(\lambda)$. Moreover, $\chi_\lambda=\chi_\mu$ when $\mu\in W\lambda$. For a fixed $\frp=\frp_I$, let $\caO^\frp$ be the category of all finitely generated $\frg$-modules $M$, which is semisimple under $\frl$-action and locally $\frp$-finite. These $M_I(\lambda)$ are the fundamental objects of $\caO^\frp$. In particular, if $I=\emptyset$, then $M(\lambda):=M_I(\lambda)$ is the Verma module with highest weight $\lambda-\rho$ and $\caO^\frp=\caO^\frb$ is the usual Bernstein-Gelfand-Gelfand category $\caO$ \cite{BGG1}. For $\lambda\in\frh^*$, let $\caO^\frp_\lambda$ be the full subcategory of $\caO^\frp$ containing modules $M$ on which $z-\chi_\lambda(z)$ acts as locally nilpotent operator for all $z\in Z(\frg)$. We also write $\caO_\lambda=\caO_\lambda^\frb$ for simplicity.

When $I\subset\Delta$ is fixed, define
\[
{}^IW:=\{w\in W\mid \ell(s_\alpha w)=\ell(w)+1\ \mbox{for all}\ \alpha\in I\}.
\]
For $\mu\in\frh^*$, set
\[
\Phi_\mu:=\{\beta\in\Phi\mid\langle\mu,\beta^\vee\rangle=0\}.
\]
The singularity of $\mu$ is measured by this subsystem $\Phi_\mu$ of $\Phi$. If $\mu$ is integral, there exists a unique anti-dominant weight $\lambda\in W\mu$. Put $J=\{\alpha\in\Delta\mid\langle\lambda, \alpha^\vee\rangle=0\}$ and
\[
W^J:=\{w\in W\mid \ell(w s_\alpha)=\ell(w)+1\ \mbox{for all}\ \alpha\in J\}.
\]
Every integral weight $\mu\in W\lambda\cap\Lambda_I^+$ can be uniquely written in the form $\mu=w_Iw\lambda$ for some $w\in{}^IW^J$, where
\[
{}^IW^J=\{w\in{}^IW\mid \ell(w)+1=\ell(ws_\alpha)\ \mbox{and}\ ws_\alpha\in {}^IW,\ \mbox{for all}\ \alpha\in J\}.
\]

Let $K(\caO)$ be the Grothendieck group of the category $\caO$. In particular, $[M]\in K(\caO)$ is the element corresponding to $M\in\caO$. The module $M$ has a composition series with simple subquotients isomorphic to some $L(\lambda)$. Denote by $[M : L(\lambda)]$ the multiplicity of $L(\lambda)$. The radical filtration of $M$ satisfying $\Rad^iM=M$ for $i\leq0$ and $\Rad^{i}M=\Rad(\Rad^{i-1}M)$ for $i>0$. Similarly, the socle filtration satisfying $\Soc^iM=0$ for $i\leq0$ and $\Soc(M/\Soc^{i}M)=\Soc^{i}M/\Soc^{i-1}M$ for $i>0$. The subquotient $\Rad_iM=\Rad^iM/\Rad^{i+1}M$ is semisimple.
%
%
\section{Graded decomposition numbers and the inverse graded decomposition numbers}
%
%

In this section, we recall the graded decomposition numbers and the inverse graded decomposition numbers for generalized Verma modules. We express those numbers explicitly in terms of the original Kazhdan-Lusztig polynomials. The results in the regular cases are already given in \cite[Theorem 3.11.4]{BGS}. The formulae we present here, which works also for the singular cases, seems not explicitly presented in any papers elsewhere.


\subsection{Koszul algebras and graded BGG resolutions}

Let $\frp=\frp_I$ be the standard parabolic subalgebra of $\frg$ as before, where $I\subset\Delta$. Let $\mu$ be an anti-dominant integral weight such that
\begin{equation}\label{paraI}
I=\bigl\{\alpha\in\Delta\bigm|\langle\mu,\alpha^\vee\rangle=0\bigr\}.
\end{equation}
Let $\lam$ be an anti-dominant integral and we define \begin{equation}\label{paraII}
J:=\bigl\{\alpha\in\Delta\bigm|\langle\lambda,\alpha^\vee\rangle=0\bigr\}.
\end{equation}
We shall denote the subcategory $\caO^\frp_\lam$ simply by $\caO^\mu_\lambda$.

The non-isomorphic simple objects in the subcategory $\caO^\mu_\lambda$ are indexed as follows: $$
\bigl\{L(w_Iw\lambda)\bigm|w\in {}^IW^J\bigr\}.
$$
For each $w\in {}^IW^J$, we use $P^\mu(w_Iw\lambda)$ to denote the projective cover of $L(w_Iw\lambda)$ in $\caO^\mu_\lambda$. Then the minimal projective generator of $\caO^\mu_\lambda$ is given by $P_\lambda^\mu:=\oplus_{w\in {}^IW^J}P^\mu(w_Iw\lambda)$. We set $$
A_{\lambda}^\mu:=\End_{\mathfrak{g}}(P_\lambda^\mu).
$$
Then we have an equivalence of categories: $\caO_\lambda^\mu\cong\text{mod-}A_{\lambda}^\mu$.

By \cite{BGS} and \cite{Bac}, we know that the $\mathbb{C}$-algebra $A_{\lambda}^\mu$ is Koszul. The corresponding graded module category (with morphisms being degree $0$ homomorphisms) is denoted by $A_{\lambda}^\mu$\text{-gmod}. For any $M,N\in A_{\lambda}^\mu$\text{-gmod}, we use $\hom_{\caO}(M,N)$ to denote the space of homomorphisms from $M$ to $N$ in $A_{\lambda}^\mu$\text{-gmod}. For each $w\in {}^IW^J$, we use $\mathbb{L}(w_Iw\lambda)$ to denote the graded lift of $L(w_Iw\lambda)$ in $A_{\lambda}^\mu$\text{-gmod} which is concentrated in degree $0$. Let $\mathbb{P}^\mu(w_Iw\lambda)$ to denote the projective cover of $\mathbb{L}(w_Iw\lambda)$ in $A_{\lambda}^\mu$\text{-gmod} which gives a graded lift of $P^\mu(w_Iw\lambda)$. We use $\mathbb{\Delta}^\mu(w_Iw\lambda)$ to denote the graded lift of $M_I(w_Iw\lambda)$ in $A_{\lambda}^\mu$\text{-gmod} such that the canonical surjection $\mathbb{\Delta}^\mu(w_Iw\lambda)\twoheadrightarrow\mathbb{L}(w_Iw\lambda)$ is a degree $0$ map.

Let $q$ be an indeterminate over $\mathbb{Z}$ and $v:=q^{1/2}$. Let $K_0(A_{\lambda}^\mu)$ be the enriched Grothendieck group of
$A_{\lambda}^\mu$\text{-gmod}, which naturally becomes a $\mathbb{Z}[v,v^{-1}]$-module via $v^kM:=M\langle k\rangle$ for any $M\in A_{\lambda}^\mu$\text{-gmod}, where $M\langle k\rangle$ is equal to $M$ upon forgetting its $\mathbb{Z}$-grading and $(M\langle k\rangle)_j:=M_{j-k}$ for any $j\in\mathbb{Z}$.

By definition, $\caO_\lambda^\mu$ is a full subcategory of $\caO_\lambda$. There is a parabolic truncation functor (i.e., Zuckerman functor) $Z_{\frp}: \caO_\lambda\rightarrow \caO_\lambda^\mu$ by $$
Z_{\frp}(M):=\text{the maximal quotient of $M$ which is locally finite over $\frp$}. $$
In particular, $$
Z_{\frp}(M(w_Iw\lambda))=\begin{cases}M_I(w_Iw\lambda), &\text{if $w\in {}^I W$;}\\
0, &\text{otherwise.}
\end{cases}
$$
Let $\epsilon: M(w_Iw\lambda)\twoheadrightarrow M_I(w_Iw\lambda)$ be the canonical surjection.

Recall that $w_I$ is the unique longest element in $W_I$. For each $1\leq k\leq m:=\ell(w_I)=|\Phi_I^+|$, we set $W_I^{k}=\{w\in W_I\ |\ \ell(w)=k\}$. For any $w\in {}^IW^J$, $w_Iw\lambda\in\Lambda_I^+$. Thus $F(\lambda-\rho)$ is a finite dimensional irreducible representation of the reductive Lie algebra $\frl_I$.
By \cite[Chapter 6]{H3}, the irreducible finite dimensional $\frl_I$-module $F(w_Iw\lambda-\rho)$ has a BGG resolution in the category of finite dimensional $U(\frl_I)$-modules (see \cite{BGG2}). Since every $\frl_I$-module can be viewed as a $\frp_I$-module with trivial $\fru_I$-action, we can apply the exact functor $U(\frg)\otimes_{U(\frp_I)}-$ on the resolution and get the exact sequence
\begin{equation}\label{flteq0}
0\rightarrow C_m\stackrel{\delta_{m}}\longrightarrow\ldots\stackrel{\delta_{k+1}}\longrightarrow
C_k\stackrel{\delta_{k}}\longrightarrow\ldots\stackrel{\delta_{1}}\longrightarrow
C_0=M(w_Iw\lambda)\stackrel{\epsilon}\longrightarrow M_I(w_Iw\lambda)\rightarrow 0,
\end{equation}
where $C_k=\bigoplus_{z\in W_I^{k}}M(zw_Iw\lambda)$.

By \cite[\S2.2]{Bac}, $A_{\lambda}^\mu$ can be realized as a $\mathbb{Z}$-graded quotient of $A_{\lambda}$ by a homogeneous ideal $\mathfrak{a}$. It follows that the Zuckerman functor $Z_\frp(-)$ allows a $\mathbb{Z}$-graded lift $\mathbb{Z}_{\frp}: M\mapsto M/\mathfrak{a}M, \forall\,M\in A_\lambda\text{\rm -gmod}$, and the surjection $M(w_Iw\lambda)\twoheadrightarrow M_I(w_Iw\lambda)$ is unique up to a scalar. Hence the surjection $\epsilon$ admits a $\mathbb{Z}$-graded lift $\hat{\epsilon}: \mathbb{\Delta}(w_Iw\lambda)\twoheadrightarrow\mathbb{\Delta}^\mu(w_Iw\lambda)$ which is homogeneous of degree zero.

For any $x,y\in W$, we have $\dim\Hom_{\caO}(M(x\lambda), M(y\lambda))\leq 1$. It follows that each map $\delta_k$ has a $\mathbb{Z}$-graded lift $\hat{\delta}_k$ which is homogeneous of degree one. As a result, we get the following lemma (compare \cite[Proposition A.2]{HM} and \cite[Appendix]{M}).

\begin{lemma}[Graded BGG resolution]\label{bgg} Let $w\in {}^IW^J$ and $\lambda$ an anti-dominant integral weight. There is an exact sequence of homomorphisms in $A_\lambda$\text{\rm -gmod}: \begin{equation}\label{flteq1}
0\rightarrow \mathbb{C}_m\stackrel{\hat{\delta}_{m}}\longrightarrow\ldots\stackrel{\hat{\delta}_{k+1}}\longrightarrow
\mathbb{C}_k\stackrel{\hat{\delta}_{k}}\longrightarrow\ldots\stackrel{\hat{\delta}_{1}}\longrightarrow
\mathbb{C}_0=\mathbb{\Delta}(w_Iw\lambda)\stackrel{\hat{\epsilon}}\longrightarrow\mathbb{\Delta}^\mu(w_Iw\lambda)\rightarrow 0,
\end{equation}
where $\mathbb{C}_k=\bigoplus_{z\in W_I^{k}}\mathbb{\Delta}(zw_Iw\lambda)$, $\hat{\epsilon}$ is homogeneous of degree zero and each map $\hat{\delta}_k$ is homogeneous of degree one.
\end{lemma}

Since $A_{\lambda}^\mu$ can be realized as a $\mathbb{Z}$-graded quotient of $A_{\lambda}$, the Grothendieck group $K_0(A_{\lambda}^\mu)$ becomes a $\mathbb{Z}[v, v^{-1}]$-submodule of $K_0(A_{\lambda})$.

\begin{cor}\label{fltprop1}
Let $w\in {}^{I}W^{J}$ and $\lambda$ an anti-dominant integral weight. Then in the Grothendieck group $K_0(A_{\lambda})$ we have
\[
[\mathbb{\Delta}^\mu(w_Iw\lambda)]=\sum_{z\in W_I}(-1)^{\ell(z)}v^{\ell(z)}[\mathbb{\Delta}(zw_Iw\lambda)].
\]
\end{cor}
\begin{proof} This follows directly from Lemma \ref{bgg}.
\end{proof}

\subsection{Graded decomposition numbers and graded inverse decomposition numbers}

For any $x,y\in W$ with $x\leq y$, we use $P_{x,y}$ to denote the corresponding Kazhdan-Lusztig polynomial introduced in \cite{KL}. For convenience, set $P_{x,y}=0$ when $x\nleq y$ as in \cite{H2}. Let $\lambda$ be an anti-dominant integral weight. By \cite[Theorem 3.11.4(ii)]{BGS}, we know that for any $x\in W^J$ we have  \begin{equation}\label{BorelSing1}
[\mathbb{\Delta}(x\lambda)]=\sum_{y\in W^J}P_{xw_0,yw_0}(v^{-2})v^{\ell(x)-\ell(y)}[\mathbb{L}(y\lambda)] . \end{equation}

The following theorem gives the graded decomposition numbers for arbitrary (possibly singular) integral blocks of the parabolic category $\caO^\frp$ (see \cite[Theorem 3.11.4(ii),(iv)]{BGS}) for the case of regular blocks).

\begin{theorem}\label{graddecomp11} Let $x\in {}^{I}W^{J}$ and $\lambda$ an anti-dominant integral weight. Then in the Grothendieck group $K_0(A_{\lambda}^\mu)$ we have \begin{equation}\label{graddecomp1}
[\mathbb{\Delta}^\mu(w_Ix\lambda)]=\sum_{y\in ^{I}W^J}\sum_{z\in W_I}(-1)^{\ell(z)}v^{\ell(x)-\ell(y)}P_{zw_Ixw_0,w_Iyw_0}(v^{-2})[\mathbb{L}(w_Iy\lambda)] .
\end{equation}
\end{theorem}

\begin{proof} Applying (\ref{BorelSing1}) to Corollary \ref{fltprop1}, we get that
\[
\begin{aligned}
\left[\mathbb{\Delta}^\mu(w_Ix\lambda)\right]&=\sum_{z\in W_I}\sum_{y\in W^J}(-1)^{\ell(z)}v^{\ell(z)+\ell(zw_Ix)-\ell(y)}P_{zw_Ixw_0,yw_0}(v^{-2})
[\mathbb{L}(y\lambda)]\\
&=\sum_{y\in W^J}\sum_{z\in W_I}(-1)^{\ell(z)}v^{\ell(w_Ix)-\ell(y)}P_{zw_Ixw_0,yw_0}(v^{-2})
[\mathbb{L}(y\lambda)] .
\end{aligned}
\]
For $y\in W^J$ and $w\in {}^{I}W^{J}$, we know that $[M_I(w_Iw\lambda):{L}(y\lambda)]\neq 0$ only if $y\in w_I {}^{I}W^J$. Thus we can restrict the summation on the righthand side of the above equality to those $y\in w_I\, ^{I}W^J$. Finally, note that $\ell(w_Ix)-\ell(w_Iy)=\ell(x)-\ell(y)$ for any $x,y\in {}^IW$. This proves the theorem.
\end{proof}

\begin{definition}\label{pq} Let $x,y\in W^J$ and $\lambda$ an anti-dominant integral weight. Define the Kazhdan-Lusztig-Vogan polynomials (\cite[Theorem 3.11.4]{BGS})
\begin{equation}\label{KLV1}
P_{x, y}^J(q)=\sum_{i\geq0}\dim\Ext_{\caO}^{i}(M(x\lambda), L(y\lambda))q^{(\ell(y)-\ell(x)-i)/2}.
\end{equation}
\end{definition}

By \cite[Theorem 3.11.4]{BGS} (and translating it into our notations), we have for any $x,y\in W^J$, $$
P_{x, y}^J(q)=\sum_{z\in W_J}(-1)^{\ell(z)}P_{xz,y}(q) .
$$

Since $A_\lambda$ is Koszul, each simple module $\mathbb{L}(y\lambda)$ has a linear injective resolution. It follows that $$
[\mathbb{L}(y\lambda)]=\sum_{x\in W^J}\Bigl(\sum_{i\geq 0}(-v)^i\dim\Ext_{\caO}^i(M(x\lambda),L(y\lambda))\Bigr)[\mathbb{\Delta}(x\lambda)].
$$
Applying \cite[Theorem 3.11.4(iv)]{BGS} we can deduce that $\Ext_{\caO}^{i}(M(x\lambda), L(y\lambda))\neq 0$ only if $i\equiv\ell(y)-\ell(x)\pmod{2}$.
As a result, we get the following lemma.

\begin{lemma}\label{inverse1} Let $y\in W^J$ and $\lambda$ an anti-dominant integral weight. Then in the Grothendieck group $K_0(A_{\lambda})$ we have $$\begin{aligned}
\left[{\mathbb{L}}(y\lambda)\right]&=\sum_{x\in W^J}(-1)^{\ell(y)-\ell(x)}v^{\ell(y)-\ell(x)} P_{x,y}^J(v^{-2})[\mathbb{\Delta}(x\lambda)]\\
 &=\sum_{x\in W^J}\sum_{z\in W_J}(-1)^{\ell(y)+\ell(z)-\ell(x)}v^{\ell(y)-\ell(x)}P_{xz,y}(v^{-2})[\mathbb{\Delta}(x\lambda)].
\end{aligned}$$
\end{lemma}

By Brauer-Humphreys reciprocity, $[{\mathbb{P}}(y\lambda):{\mathbb{\Delta}}(x\lambda)\langle k\rangle]=[{\mathbb{\Delta}}(x\lambda):{\mathbb{L}}(y\lambda)\langle k\rangle]$ for any $k\in\mathbb{Z}$. We can deduce from Lemma \ref{inverse1} that for any $x\in W^J$, \begin{equation}\label{inverse2}
[{\mathbb{\Delta}}(x\lambda)] =\sum_{x\in W^J}\sum_{z\in W_J}(-1)^{\ell(y)+\ell(z)-\ell(x)}v^{\ell(y)-\ell(x)}P_{xz,y}(v^{-2})
[{\mathbb{P}}(y\lambda)] .
\end{equation}

The following theorem gives the graded inverse decomposition numbers for arbitrary integral blocks of the parabolic category $\caO^\frp$ (see \cite[Theorem 3.11.4(i),(iv)]{BGS} for the case of regular blocks). It seems that this might be known to some experts (cf. \cite[Appendix]{CPS}), but we couldn't find a suitable reference anywhere.

\begin{theorem}\label{inversegraddecomp11} Let $y\in {}^{I}W^{J}$  and $\lambda$ an anti-dominant integral weight. Then in the Grothendieck group $K_0(A_{\lambda}^\mu)$ we have
\begin{equation*}\label{inverse22}
[{\mathbb{L}}(w_Iy\lambda)] =\sum_{x\in ^{I}W^J}\sum_{z\in W_J}(-1)^{\ell(y)+\ell(z)-\ell(x)}v^{\ell(y)-\ell(x)}P_{w_Ixz,w_Iy}(v^{-2})
[{\mathbb{\Delta}}^\mu(w_Ix\lambda)] .
\end{equation*}
\end{theorem}

\begin{proof} Let $x\in {}^{I}W^{J}$. Using \cite[Theorem 4.14]{CM} and applying the derived Zuckerman functor to (\ref{inverse2}) with $x$ replaced by $w_Ix$, we get that \begin{equation*}\label{inverse3}
[{\mathbb{\Delta}}^\mu(w_Ix\lambda] = \sum_{x\in ^{I}W^J}\sum_{z\in W_J}(-1)^{\ell(y)+\ell(z)-\ell(w_Ix)}v^{\ell(y)-\ell(w_Ix)}P_{w_Ixz,y}(v^{-2})
[\mathbb{Z}_{\frp}({\mathbb{P}}(y\lambda))] .
\end{equation*}

Note that $\mathbb{Z}_{\frp}({\mathbb{P}}(y\lambda))\neq 0$ only if $y\in w_I {}^IW^{J}$, and $\mathbb{Z}_{\frp}({\mathbb{P}}(w_Iw\lambda))\cong
{\mathbb{P}}^\mu(w_Iw\lambda)$ for $w\in{}^IW^J$. It follows that
$$
[{\mathbb{\Delta}}^\mu(w_Ix\lambda)] =\sum_{x\in ^{I}W^J}\sum_{z\in W_J}(-1)^{\ell(w)+\ell(z)-\ell(x)}v^{\ell(w)-\ell(x)}P_{w_Ixz,w_Iw}(v^{-2})
[{\mathbb{P}}^\mu(w_Iw\lambda)] .
$$
Equivalently, we get that $$
[{\mathbb{L}}(w_Iy\lambda)] =\sum_{x\in ^{I}W^J}\sum_{z\in W_J}(-1)^{\ell(y)+\ell(z)-\ell(x)}v^{\ell(y)-\ell(x)}P_{w_Ixz,w_Iy}(v^{-2})
[{\mathbb{\Delta}}^\mu(w_Ix\lambda)] .
$$
\end{proof}

\begin{definition} Let $x,y\in {}^IW^J$. We define $$
{}^IP_{x, y}^J(q) = \sum_{z\in W_J}(-1)^{\ell(z)}P_{w_Ixz, w_Iy}(q),\quad\,
{}^IQ_{x, y}^J(q)  = \sum_{z\in W_I}(-1)^{\ell(z)}P_{zw_Ixw_0,w_Iyw_0}(q) .
$$
\end{definition}

It follows from Theorems \ref{graddecomp11} and \ref{inversegraddecomp11} that $$\begin{aligned}
{[}{\mathbb{\Delta}}^{\mu}(w_Ix\lambda){]} & =\sum_{y\in ^{I}W^J}v^{\ell(x)-\ell(y)}{}^IQ_{x, y}^J(v^{-2})[\mathbb{L}(w_Iy\lambda)] ,\\
[{\mathbb{L}}(w_Iy\lambda)] & = \sum_{x\in ^{I}W^J}(-1)^{\ell(y)-\ell(x)}v^{\ell(y)-\ell(x)}{}^IP_{x, y}^J(v^{-2}) [{\mathbb{\Delta}}^\mu(w_Ix\lambda)] .
\end{aligned}
$$
As a result, we get the following corollary.

\begin{cor}\label{gdcor1}
Let $I, J\subset\Delta$ and $x, y\in{}^IW^J$. Then
\[
\sum_{z\in {}^IW^J, x\leq z\leq y}(-1)^{\ell(z)+\ell(y)}{}^IP^J_{x, z}(q){}^IQ^J_{y,z}(q)=\delta_{x, y} .
\]
\end{cor}

Specializing $v$ to $1$, we can get the ungraded decomposition number and inverse decomposition number of the parabolic category $\caO^\frp$ as follows: $$\begin{aligned}
{[} M_I(w_Ix\lambda) {]} & =\sum_{y\in ^{I}W^J}\sum_{z\in W_I}(-1)^{\ell(z)}P_{zw_Ixw_0,w_Iyw_0}(1) [L(w_Iy\lambda)] ,\\
[L(w_Iy\lambda)] & =\sum_{y\in ^{I}W^J}\sum_{z\in W_J}(-1)^{\ell(y)-\ell(x)+\ell(z)}P_{w_Ixz, w_Iy}(1)[M_I(w_Ix\lambda)] .
\end{aligned}
$$

The polynomial ${}^IP_{x, y}^J(q)$ is the so-called generalized Kazhdan-Lusztig-Vogan polynomials (\cite{BH}, \S9.2). That says,
\begin{equation}\label{ikleq1}
{}^IP_{x, y}^J(q)=\sum_{i\geq0}q^{(\ell(y)-\ell(x)-i)/2}\dim\Ext_{\caO^\frp}^i(M_I(w_Ix\mu), L(w_Iy\mu)).
\end{equation}
In particular, $P_{x, y}:={}^\emptyset P_{x, y}^\emptyset$ is equal to the ordinary Kazhdan-Lusztig polynomial (see for example Theorem 8.11 in \cite{H3}), while ${}^IP_{x, y}:={}^IP_{x, y}^\emptyset$ is equal to the relative (or parabolic) Kazhdan-Lusztig polynomials (\cite{CC}). We also set $P_{x, y}^J:={}^\emptyset P_{x, y}^J$. Similarly set $Q_{x, y}:={}^\emptyset Q_{x, y}^\emptyset$, ${}^IQ_{x, y}:={}^IQ_{x, y}^\emptyset$ and $Q_{x, y}^J:={}^\emptyset Q_{x, y}^J$.

\begin{lemma}[\cite{Dy, Lu, De, So1, I2, BH}]\label{iklprop1}
Let $I, J\subset\Delta$. Then
\begin{itemize}
\item [(1)] $P_{x, y}=P_{x^{-1}, y^{-1}}=P_{w_0xw_0, w_0yw_0}$, where $x, y\in W$.
\item [(2)] ${}^IP_{x, y}=P_{w_Ix, w_Iy}$, where $x, y\in {}^IW$.
\item [(1)] $Q_{x, y}=P_{xw_0,yw_0}=P_{w_0x,w_0y}$, where $x, y\in W$.
\item [(2)] $Q_{x, y}^J=Q_{x, y}$, where $x, y\in W^J$.
\end{itemize}
\end{lemma}

Let $x\in {}^IW^J$ and $\lambda$ an anti-dominant integral weight. By \cite{BGS} and \cite{Bac}, $A_\lambda^\mu$ is Koszul. The grading filtration on $\mathbb{\Delta}^\mu(w_Ix\lambda)$ coincides with its radical filtration up to a shift. Hence we get the following corollary.

\begin{cor}\label{radcor} Let $\lambda$ be an anti-dominant integral weight and $x, y\in {}^IW^J$. Then we have
\begin{equation}\label{ikleq3}
{}^IQ_{x, y}^J(q)=\sum_{i\geq0}[\Rad_iM_I(w_Ix\lambda) : L(w_Iy\lambda)]q^{(\ell(x)-\ell(y)-i)/2}.
\end{equation}
\end{cor}

We need the following result about ${}^IW^J$.

\begin{prop}[{\cite[Proposition 2.4.2]{BN}}]\label{bnprop1}
Let $I, J\subset\Delta$.
\begin{itemize}
  \item [(\rmnum{1})] There is a bijection ${}^If^J: {}^IW^J\ra {}^JW^{-w_0I}$ given by ${}^If^J(w)=w_Jw^{-1}w_Iw_0$.
  \item [(\rmnum{2})] There is a bijection ${}^Ig^J: {}^IW^J\ra {}^{-w_0J}W^{I}$ given by ${}^Ig^J(w)=w_0w_Jw^{-1}w_I$.
\end{itemize}
\end{prop}

\begin{cor}\label{iklcor1}
Let $x, y\in {}^IW^J$. Then
\[
{}^IQ^J_{x,y}={}^{-w_0J}P^I_{{}^Ig^J(x),{}^Ig^J(y)}.
\]
\end{cor}
\begin{proof}
Let $J'=-w_0J$. Then $J'\subset\Delta$. By Lemma \ref{iklprop1} and Definition \ref{pq} we have
\[
\begin{aligned}
{}^IQ_{x, y}^J=&\sum_{z\in W_I}(-1)^{\ell(z)}P_{zw_Ixw_0, w_Iyw_0}=\sum_{z\in W_I}(-1)^{\ell(z)}P_{w_0x^{-1}w_Iz, w_0y^{-1}w_I}\\
=&\sum_{z\in W_I}(-1)^{\ell(z)}P_{w_{J'}w_0w_Jx^{-1}w_Iz, w_{J'}w_0w_Jy^{-1}w_I}\\
=&\sum_{z\in W_I}(-1)^{\ell(z)}P_{w_{J'}{}^Ig^J(x)z, w_{J'}{}^Ig^J(y)}={}^{J'}p^I_{{}^Ig^J(x),{}^Ig^J(y)}.
\end{aligned}
\]
\end{proof}

By Definition \ref{pq}, ${}^IQ_{x, y}^J$ is a polynomial in $q$. Combining this with Corollary \ref{radcor}, we obtain the following parity property on the  radical filtration of generalized Verma modules.

\begin{lemma}\label{ikllem2}
Let $x, y\in {}^IW^J$  and $\lambda$ an anti-dominant integral weight. The radical filtrations of generalized Verma modules satisfy the parity property, that is, $[\Rad_iM_I(w_Ix\lambda), L(w_Iy\lambda)]=0$ unless $\ell(x)-\ell(y)\equiv i (\Mod\ 2)$.
\end{lemma}

Suppose that $\nu\in\Lambda$ is a regular anti-dominant integral weight. Let $T_{\nu}^{\lambda}(-)$, $T_{\lambda}^{\nu}(-)$ be the translation functors as defined in \cite{J2} and \cite[\S7.1]{H3}. By \cite{Str}, we know that those translation functors allow $\mathbb{Z}$-graded lifts. Let $\Theta_\nu^\lambda(-)$, $\Theta_\lambda^\nu(-)$ be their corresponding graded lift respectively. We choose these graded lifts in the same way as \cite{Str} and \cite[\S2.2]{CM}. In particular,
\begin{equation}\label{gdeq1}
\Theta_\nu^\lambda(\mathbb{L}(x\nu))=\begin{cases}\mathbb{L}(x\lambda)\langle -\ell(w_J)\rangle,\quad&\text{for}\ x\in W^J;\\
0, &\text{for}\ x\not\in W^J.
\end{cases}
\end{equation}
We also have $\Theta_\lambda^\nu(\mathbb{P}(x\lambda))=\mathbb{P}(x\nu)$ for $x\in W^J$.

\begin{prop}[{\cite[Theorem 4.3]{CM}}]\label{gdprop1}
Let $\lambda,\nu$ be defined as above and let $z\in {}^IW$. Then
\[
\Theta^\lambda_\nu(\mathbb{\Delta}^\mu(w_Iz\nu))=\begin{cases}\mathbb{\Delta}^\mu(w_Iy\lambda)\langle \ell(x)-\ell(w_J)\rangle, &\text{if $z=yx$, $y\in {}^IW^J, x\in W_J$};\\
0, &\text{otherwise}.
\end{cases}
\]
\end{prop}
\begin{proof}
We need to explain the notation difference. First assume that $z=yx$ for $y\in {}^IW^J$ and $x\in W_J$. Denote $y'=w_Iyw_0$ and $x'=w_0xw_0$. Thus $x'\in W_{J'}$, where $J'=-w_0J$. Proposition \ref{bnprop1} implies that $y'\in {}^IW^{J'}w_{J'}$. It is easy to see that $\ell(x')=\ell(x)$ and $\ell(w_{J'})=\ell(w_{J})$. Set $z'=w_Izw_0$, $\nu'=w_0\nu$ and $\lambda'=w_0\lambda'$. Then $\nu'$ is a regular dominant weight and $\lambda'$ is a dominant weight with $\Phi_{\lambda'}=\Phi_{J'}$. With $z'\nu'=w_Iz\nu$ and $z'=y'x'$, \cite[Theorem 4.3]{CM} yields
\[
\Theta^\lambda_\nu(\mathbb{\Delta}^\mu(w_Iz\nu))=\Theta^\lambda_\nu(\mathbb{\Delta}^\mu(z'\nu'))=\mathbb{\Delta}^\mu(y'\lambda')\langle \ell(x')-\ell(w_{J'})\rangle=\mathbb{\Delta}^\mu(w_Iy\lambda)\langle \ell(x)-\ell(w_J)\rangle.
\]
The above argument also shows that $z=yx$ for $y\in {}^IW^J$ and $x\in W_J$ if and only if $z'=y'x'$ for $y'\in {}^IW^{J'}w_{J'}$ and $x'\in W_{J'}$, which implies the second part of the Proposition.
\end{proof}

We remark that Theorem \ref{inversegraddecomp11} can also be deduced by applying Proposition \ref{gdprop1} to the formulae of the graded inverse decomposition numbers in the regular case (\cite[Theorem 3.11.4]{BGS}).

Define a map $t_\nu^\lambda: K(\caO_\nu)\ra K(\caO_\lambda)$ such that $t_\nu^\lambda([M])=[T_\nu^\lambda M]$ for $M\in\caO_\nu$. The exactness of $T_\nu^\lambda$ insures that $t_\nu^\lambda$ is a linear transformation on $K(\caO_\nu)$. Thus (\ref{gdeq1}) and Proposition \ref{gdprop1} gives the following corollary, which we need in the next section.

\begin{cor}\label{gdcor2}
Let $\lambda,\nu$ be defined as above. Then
\begin{itemize}
  \item [(1)] $t^\lambda_\nu([\Rad^i M(z\nu)])=[\Rad^{i}M(z\lambda)]$ for $i\in\bbZ$ and $z\in W^J$.
  \item [(2)] $t^\lambda_\nu([M(z\nu)])=[M(z\lambda)]$ for $z\in W$.
\end{itemize}
\end{cor}

\begin{proof} (1) We consider the graded functor $\mathcal{T}:=\Theta^\lambda_\nu\langle\ell(w_J)\rangle$. Applying Proposition \ref{gdprop1}, we see that
$\mathcal{T}(\mathbb{\Delta}(z\nu))=\mathbb{\Delta}(z\lambda)$ and $\mathcal{T}(\mathbb{L}(z\nu))=\mathbb{L}(z\lambda)$. Since $\mathcal{T}$ is a $\mathbb{Z}$-graded functor and the grading filtrations of $\mathbb{\Delta}(z\nu)$ and $\mathbb{\Delta}(z\lambda)$ coincide with their radical filtration up to a shift, it follows that for any $j\geq 0$, $$
\mathcal{T}\bigl(\Rad^i(\mathbb{\Delta}(z\nu))\bigr)=\mathcal{T}\bigl(\oplus_{j\geq i}\mathbb{\Delta}(z\nu)_j\bigr)\subseteq \oplus_{j\geq i}\mathbb{\Delta}(z\lambda)_j=\Rad^i(\mathbb{\Delta}(z\lambda)).
$$
Combining this with the equality $\mathcal{T}(\mathbb{\Delta}(z\nu))=\mathbb{\Delta}(z\lambda)$ and the fact that $\mathcal{T}$ is an exact functor, we can deduce that all the above inclusion are actually equalities. This proves $\mathcal{T}\bigl(\Rad^i(\mathbb{\Delta}(z\nu))\bigr)=\Rad^i(\mathbb{\Delta}(z\lambda))$ which implies (1).

(2) This is an easy consequence of Proposition \ref{gdprop1} or Theorem 7.6 in \cite{H3}.
\end{proof}

The above results can also be found in \cite{I2}.

%
%
\section{The sum formula of radical filtrations}
%
%

In this section, we shall give a sum formula about radical filtrations of generalized Verma modules. It can be viewed as generalization of the Jantzen sum formula for Verma modules. For $\mu, \nu\in\frh^*$, we write $\nu\leq\mu$ if $\Hom_\caO(M(\nu), M(\mu))\neq0$. This gives a partial ordering which can be viewed as the Bruhat ordering on $\frh^*$ (\cite[\S2]{ES}).

Since each Verma module in the regular block of the usual BGG category $\caO$ is rigid, its Jantzen filtration coincides with its grading filtration up to a shift \cite{BB}. When $\mu\in\frh^*$ is regular, the following proposition is essentially equivalent to Jantzen sum formula (\cite{J2}) for Verma module $M(\mu)$ in the usual BGG category $\caO$. We need more effort when $\mu$ is singular.

\begin{prop}\label{sfprop1}
Let $\mu\in\frh^*$.
\[
\sum_{i>0}[\Rad^{i}M(\mu)]=\sum_{\beta\in\Phi^+, s_\beta \mu<\mu}[M(s_\beta \mu)].
\]
\end{prop}

\begin{proof}
Obviously $s_\beta \mu<\mu$ is equivalent to $\langle\mu, \beta^\vee\rangle\in\bbZ^{>0}$ (see for example \cite{H3}). With Soergel's result (\cite[Theorem 11]{So2}, see \cite[Theorem 11.13]{H3} for a English translation), it suffices to consider the integral case. When $\mu$ is regular, the Jantzen filtration of $M(\mu)$ coincides with its radical filtration \cite{BB}. This is exactly the Jantzen sum formula \cite{J2, B} for regular Verma modules. If $\mu$ is singular, there exists anti-dominant weight $\lambda\in W\mu$ so that $\mu=w\lambda$ for some $w\in W^J$, where $J=\{\alpha\in\Delta\mid\langle\lambda, \alpha^\vee\rangle=0\}$. Let $\nu$ be a regular integral anti-dominant weight (e.g., $\nu=-\rho$). In view of Corollary \ref{gdcor2}, we obtain
\[
\begin{aligned}
\sum_{i>0}[\Rad^{i}M(\mu)]=&\sum_{i>0}t_{\nu}^{\lambda}([\Rad^{i}M(w\nu)])\\
=&\sum_{\beta\in\Phi^+, s_\beta w\nu<w\nu}t_{\nu}^{\lambda}([M(s_\beta w\nu)])\\
=&\sum_{\beta\in\Phi^+, s_\beta w\nu<w\nu}[M(s_\beta w\lambda)]\\
=&\sum_{\beta\in\Phi^+, s_\beta \mu<\mu}[M(s_\beta\mu)],
\end{aligned}
\]
where the second equality follows from the regular case, the fourth equality follows from the fact that $s_\beta w\nu<w\nu$ if and only if $w^{-1}\beta<0$ and if and only if $s_\beta w\lambda<w\lambda$. To see this, it suffices to show that $s_\beta w\lambda\neq w\lambda$ for $\beta\in\Phi^+$ when $s_\beta w\nu<w\nu$. Otherwise $\beta=w\alpha$ for some $\alpha\in\Phi_J^+$ (note that $w\alpha>0$ for $w\in W^J$). This forces $0<\langle w\nu,\beta\rangle=\langle\nu, \alpha\rangle<0$, a contradiction.
\end{proof}

\begin{lemma}\label{sflem1}
Let $i\in\bbZ^{\geq0}$ and $\mu\in\Lambda_I^+$. Then
\[
[\Rad^{i}M_I(\mu)]=\sum_{w\in W_I}(-1)^{\ell(w)}[\Rad^{i-\ell(w)}M(w\mu)].
\]
\end{lemma}
\begin{proof}
We only need to consider the integral case in view of Soergel's category equivalence (\cite[Theorem 11]{So2}). Then the lemma is an immediate consequence of Proposition \ref{fltprop1}.
\end{proof}

Setting $i=0$ in Lemma \ref{sflem1}, we obtain the following result (which also follows from Corollary \ref{fltprop1}).
\begin{equation}\label{sfeq1}
[M_I(\mu)]=\sum_{w\in W_I}(-1)^{\ell(w)}[M(w\mu)].
\end{equation}
This is Proposition 9.6 in \cite{H3}.

\begin{definition} For any $\mu\in\frh^*$, we define $$
\theta(\mu):= \sum_{w\in W_I}(-1)^{\ell(w)}[M(w\mu)] .
$$
\end{definition}
These are the character formulae defined in \cite{J1}, which can be used to determine the simplicity of generalized Verma modules.

\begin{prop}[\cite{J2, Mat1, Ku}]\label{jsprop1}
Let $\mu\in\frh^*$.
\begin{itemize}
\item [(1)] $\theta(w\mu)=(-1)^{\ell(w)}\theta(\mu)$ for $w\in W_I$.
\item [(2)] If $\langle \mu,\alpha^\vee\rangle=0$ for some $\alpha\in\Phi_I$, then $\theta(\mu)=0$.
\item [(3)] If $\langle \mu,\alpha^\vee\rangle\in\bbZ\backslash\{0\}$ for all $\alpha\in\Phi_I$, there exists $w\in W_I$ so that $w\mu\in\Lambda_I^+$ and $\theta(\mu)=(-1)^{\ell(w)}[M_I(w\mu)]$.
\end{itemize}
\end{prop}

\begin{lemma}\label{sflem2}
Let $\mu\in\Lambda_I^+$. Then
\begin{equation*}
\sum_{w\in W_I}(-1)^{\ell(w)}\sum_{\beta\in\Phi_I^+, s_\beta w\mu<w\mu}[M(s_\beta w\mu)]
=\sum_{w\in W_I}(-1)^{\ell(w)+1}\ell(w)[M(w\mu)].
\end{equation*}
\end{lemma}
\begin{proof}
For any $w\in W_I$ and $\beta\in\Phi_I^+$ with $s_\beta w\mu<w\mu$, denote $w'=s_\beta w$. Then $(w')^{-1}\beta\in\Phi_I$ and
\[
\langle\mu, ((w')^{-1}\beta)^\vee\rangle=\langle w'\mu, \beta^\vee\rangle=\langle s_\beta w\mu, \beta^\vee\rangle=-\langle w\mu, \beta^\vee\rangle\in\bbZ^{<0}.
\]
One has $(w')^{-1}\beta<0$ since $\mu\in\Lambda_I^+$. The number of $\beta\in\Phi_I^+$ with such a property is exactly $\ell(w')$ (see for example \cite{H2}). Since $l(w)\equiv l(w')+1\pmod{2}$, we obtain
\[
\sum_{w\in W_I}(-1)^{\ell(w)}\sum_{\beta\in\Phi_I^+, s_\beta w\mu<w\mu}[M(s_\beta w\mu)]
=\sum_{w'\in W_I}(-1)^{l(w')+1}\ell(w')[M(w'\mu)].
\]
\end{proof}

\begin{definition} We define $$\Psi_\mu^+=\{\beta\in\Phi^+\backslash\Phi_I^+\mid \langle\mu, \beta^\vee\rangle\in\bbZ^{>0}\}. $$
\end{definition}

The following theorem gives a sum formula for the radical filtration of the generalized Verma module, which can be viewed as a generalization of Proposition \ref{sfprop1}.

\begin{theorem}\label{sfthm1}
Let $\mu\in\Lambda_I^+$. Then
\begin{equation}\label{sft1eq1}
\sum_{i\geq 1}[\Rad^{i}M_I(\mu)]=\sum_{\beta\in \Psi_\mu^+}\theta(s_\beta\mu).
\end{equation}

\end{theorem}
\begin{proof}
Note that $\langle\mu, \beta^\vee\rangle=\langle w\mu, (w\beta)^\vee\rangle$. Thus for any $w\in W_I$, $\beta\in\Psi_\mu^+$ if and only if $w\beta\in\Psi_{w\mu}^+$. In view of Lemma \ref{sflem1}, one has
\[
\begin{aligned}
&\sum_{i>0}[\Rad^iM_I(\mu)]\\
=&\sum_{i>0}\sum_{w\in W_I}(-1)^{\ell(w)}[\Rad^{i-\ell(w)}M(w\mu)]\\
=&\sum_{w\in W_I}(-1)^{\ell(w)}\sum_{i>\ell(w)}[\Rad^{i-\ell(w)}M(w\mu)]+\sum_{w\in W_I}(-1)^{\ell(w)}\ell(w)[M(w\mu)]\\
=&\sum_{w\in W_I}(-1)^{\ell(w)}\sum_{\beta\in\Phi^+, s_\beta w\mu<w\mu}[M(s_\beta w\mu)]+\sum_{w\in W_I}(-1)^{\ell(w)}\ell(w)[M(w\mu)]\\
=&\sum_{w\in W_I}(-1)^{\ell(w)}\sum_{\beta\in\Psi_{w\mu}^+}[M(s_\beta w\mu)]\\
=&\sum_{w\in W_I}(-1)^{\ell(w)}\sum_{w^{-1}\beta\in\Psi_{\mu}^+}[ M(ws_{w^{-1}\beta}\mu)]\\
=&\sum_{\gamma\in \Psi_\mu^+}\theta(s_\gamma\mu),
\end{aligned}
\]
where the third equality follows from Proposition \ref{sfprop1} and the fourth equality follows from Lemma \ref{sflem2}.
\end{proof}

Note that $M_I(\mu)$ is simple if and only if $\Rad^{i}M_I(\mu)=0$ for $i\geq1$. With Theorem \ref{sfthm1}, we can recover the famous Jantzen's simplicity criteria for generalized Verma modules obtained in \cite{J1}.

\begin{theorem}[{\cite[Collar 1]{J1}}]
Let $\mu\in\Lambda_I^+$. Then
$M_I(\mu)$ is simple if and only if
\[
\sum_{\beta\in\Psi_\mu^+}\theta(s_\beta\mu)=0.
\]
\end{theorem}

\begin{remark}\label{lormk2}
If $\mu$ is regular, (\ref{sft1eq1}) can be found in \cite[Corollary 7.1.4]{I3}. The Jantzen filtrations for Verma modules in the usual BGG category $\caO$ are well-known in the literature. As pointed out in Remark 9.17 of \cite{H3}, Jantzen introduced a similar ``Jantzen filtration'' for generalized Verma modules in \cite{J2}. In fact, with Lemma 3, Satz 2 and the observation in the Bemerkung before Lemma 4 in \cite{J2}, along Jantzen's line for Verma modules (see \cite{J1} or \cite{H3}), one can obtain the following generalization: Let $\mu\in\Lambda_I^+$. Then $M_I(\mu)$ has a filtration by submodules
\[
M_I(\mu)=M_I(\mu)^0\supset M_I(\mu)^1\supset M_I(\mu)^2\supset\ldots
\]
with $M_I(\mu)^i=0$ for large $i$, such that
\begin{itemize}
\item [(1)] Every nonzero quotient $M_I(\mu)^i/M_I(\mu)^{i+1}$ has a nondegenerate contravariant form.

\item [(2)] $M_I(\mu)^1$ is the unique maximal submodule of $M_I(\mu)$.

\item [(3)] There is a formula:
\[
\sum_{i>0}[M_I(\mu)^i]=\sum_{\beta\in\Psi_\mu^+}\theta(s_\beta\mu).
\]
\end{itemize}
With (\ref{sft1eq1}), one might expect $M_I(\mu)^i=\Rad^iM_I(\mu)$, i.e., the Jantzen filtration of $M_I(\mu)$ coincides with its radical filtration (though $M_I(\mu)$ is in general not rigid anymore). Our Theorem \ref{sfthm1} gives a strong evidence in support of this speculation. This seems to be true when $\mu$ is regular (\cite{BB, Sh}). The singular case is not known.
\end{remark}
%
%
\section{Radical filtration for generalized Verma modules}
%
%

In this section, we will use the sum formula (\ref{sft1eq1}) and other results to determine the radical filtration of basic generalized Verma modules defined in \cite{XZ}.

\subsection{Jantzen coefficients and Gelfand-Kirillov dimension} First recall the Jantzen coefficients defined in \cite{XZ}.

\begin{definition} Let $\mu\in\Lambda_I^+$. We can write (see Proposition \ref{jsprop1})
\begin{equation}\label{jceq1}
\sum_{\beta\in\Psi_\mu^+}\theta(s_\beta\mu)=\sum_{\mu>\nu\in\Lambda_I^+}c(\mu, \nu)[M_I(\nu)],
\end{equation}
where for each $\nu\in\Lambda_I^+$, $c(\mu, \nu)\in\mathbb{Z}$ is called the \emph{Jantzen coefficients} associated to $(\mu,\nu)$.
\end{definition}

The Jantzen coefficients can be calculated through a reduction process (\cite[\S4]{XZ}). In particular, given $\mu\in\Lambda_I^+$, $c(\mu, \nu)$ is nonzero for only finitely many $\nu\in\Lambda_I^+$. The following result follows directly from Proposition \ref{jsprop1}.

\begin{lemma} Let $\mu, \nu\in\Lambda_I^+$ with $\mu>\nu$. Set
\begin{equation}\label{jfeq3}
\Psi_{\mu, \nu}^+:=\{\beta\in\Psi_\mu^+\mid \nu=w_\beta s_\beta\mu\ \mbox{for some}\ w_\beta\in W_I\}.
\end{equation}
Then $c(\mu, \nu)=\sum_{\beta\in \Psi_{\mu, \nu}^+}(-1)^{\ell(w_\beta)}$.
\end{lemma}

\begin{example}\label{sjex1}
Let $\frg=\frsl(3,\bbC)$ and $I=\{e_1-e_2\}$. Put $\mu=(1, 0, -1)$, $\nu=(1, -1, 0)$ and $\zeta=(0, -1, 1)$. Then $\Psi_\mu^+=\{e_1-e_3, e_2-e_3\}$. So
\[
\sum_{\beta\in\Psi_\mu^+}\theta(s_\beta\mu)=\theta(\nu)+\theta(s_{e_1-e_2}\zeta)=[M_I(\nu)]-[M_I(\zeta)].
\]
Therefore $c(\mu, \nu)=1$ and $c(\mu, \zeta)=-1$. Similarly $c(\nu, \zeta)=1$, while the other Jantzen coefficients are vanished.
\end{example}

Recall that $\Rad_{i}M:=\Rad^{i}M/\Rad^{i+1}M$ for each $M\in\caO$ and $i\in\mathbb{Z}$. Theorem \ref{sfthm1} implies the following result.
\begin{cor}\label{jccor1}
Let $\mu\in\Lambda_I^+$. Then
\begin{equation}\label{jceq2}
\sum_{i\geq 1}i[\Rad_{i}M_I(\mu)]=\sum_{\mu>\nu\in\Lambda_I^+}c(\mu, \nu)[M_I(\nu)].
\end{equation}
\end{cor}

The following example shows how to use (\ref{jceq2}) to get the radical filtration of a generalized Verma modules.

\begin{example}\label{sjex2}
Using notations in Example \ref{sjex1}, (\ref{jceq2}) yields
\[
\sum_{i\geq 1}i[\Rad_{i}M_I(\zeta)]=0.
\]
This forces $\Rad_{i}M_I(\zeta)=0$ for $i\geq 1$, that is, $M_I(\zeta)=\Rad_{0}M_I(\zeta)=L(\zeta)$. Similarly,
\[
\sum_{i\geq 1}i[\Rad_{i}M_I(\nu)]=c(\nu, \zeta)[M_I(\zeta)]=[L(\zeta)].
\]
We must have $\Rad_{1}M_I(\nu)=L(\zeta)$ and $\Rad_{i}M_I(\nu)=0$ for $i>1$. At last,
\[
\sum_{i\geq 1}i[\Rad_{i}M_I(\mu)]=c(\mu, \nu)[M_I(\nu)]+c(\mu, \zeta)[M_I(\zeta)]=[L(\nu)]
\]
implies $\Rad_{1}M_I(\mu)=L(\nu)$ and $\Rad_{i}M_I(\nu)=0$ for $i>1$. To summarize:
\[
M_I(\mu)=\begin{aligned}
&L(\mu)\\
&L(\nu)
\end{aligned} \qquad
M_I(\nu)=\begin{aligned}
&L(\nu)\\
&L(\zeta)
\end{aligned} \qquad
M_I(\zeta)=\begin{aligned}
&L(\zeta)
\end{aligned}
\]
\end{example}

In order to determine the radical filtration of more generalized Verma modules, we might need the following results about Gelfand-Kirillov dimension. Details can be found in \cite{V1}. Suppose $M$ is a $U(\mathfrak{g})$-module generated by a finite-dimensional subspace $M_0$. For each $n\in\mathbb{N}$, we define $\varphi_{M,M_0}(n)=\dim( U_n(\mathfrak{g})M_{0})$, where $U_{n}(\mathfrak{g})$ is the $\mathbb{C}$-subspace of the universal enveloping algebra $U(\mathfrak{g})$ of $\frg$ spanned by all the products $y_1y_2\cdots y_s$ with $s\leq n$ and $y_i\in\mathfrak{g}$ for each $i$.

\begin{lemma}[{\cite[Lemma 2.1]{V1}}]\label{gklem1}
There exists a unique polynomial $\overline{\varphi}_{M,M_0}(v)\in\mathbb{Q}[v]$ such that $\overline{\varphi}_{M,M_0}(n)=\varphi_{M,M_0}(n)$ for large $n$. The leading term
of $\overline{\varphi}_{M,M_0}(v)$ is $\frac{c(M)}{(d_{M})!}v^{d_{M}}$, where $c(M)$, $d_M$ are nonnegative integers independent of $M_0$.
\end{lemma}

The integer $d_M$ is the \emph{Gelfand-Kirillov dimension} of $M$. We write $d_M=\GK(M)$, while $c(M)$ is called the {\it Bernstein degree} of $M$. For $d\in\bbZ^{\geq0}$, denote by $c_d(M)$ the coefficient of $v^d/d!$ in the polynomial $\overline{\varphi}_{M,M_0}(v)$.

\begin{lemma}[{\cite[Lemma 2.4]{V1}}]\label{gklem2}
Suppose $0\rightarrow A\rightarrow B\rightarrow C\rightarrow 0$ is an exact sequence of finitely generated $U(\mathfrak{g})$ modules $A, B, C$. Then
$d_B=\max(d_A,d_C)$ and $c_d(B)=c_d(A)+c_d(C)$, where $d=d_B$.
\end{lemma}

The following result is well known.
\begin{lemma}\label{gklem3}
For $\lambda\in\Lambda_I^+$, set $d=\GK(M_I(\lambda))$. Then $d=|\Phi^+\backslash\Phi_I|$ and $c_d(M_I(\lambda))=\dim F(\lambda-\rho)$.
\end{lemma}

When $I$ is fixed, a weight $\mu\in\Lambda_I^+$ is called \emph{socular} if it appears as a summand in the socle of some generalized Verma modules. The following result can be found in \cite[\S1, \S4.6]{I3}.

\begin{lemma}\label{gklem4}
Let $\mu\in\Lambda_I^+$. Then
\begin{itemize}
\item [(1)] $\mu$ is socular if and only if $\GK(L(\mu))=|\Phi^+\backslash\Phi_I|$.
\item [(2)] If $\mu$ is socular, the set $\{\zeta\in\Lambda_I^+\mid [M_I(\zeta): L(\mu)]>0\}$ contains a unique maximal element $m(\mu)$. Moreover, $M_I(m(\mu))$ has simple socle $L(\mu)$ and $[M_I(m(\mu)): L(\mu)]=1$.
\end{itemize}
\end{lemma}

From now on in this section, we will apply the previous results to obtain the radical filtration of basic generalized Verma modules.

\subsection{basic generalized Verma modules}
Let $\Phi$ be an irreducible system and denote by $\Delta=\{\alpha_1, \ldots, \alpha_n\}$ the simple roots corresponding to the standard numbering of vertices in the Dynkin diagram of $\Phi$ (\cite[\S 11.4]{H1}). Let $\lambda$ be an integral anti-dominant weight with $J=\{\alpha\in\Delta\mid\langle\lambda, \alpha^\vee\rangle=0\}$. Denote $K=-w_0J$ and $\nu=w_0\lambda$. Then $K\subset\Delta$ and $\nu$ is dominant. If $\rank\ \Phi_I=\rank\ \Phi_K=\rank\ \Phi-1$, then we call each $M_I(\mu)$ with $\mu\in W\lambda\cap\Lambda_I^+$ a \emph{basic generalized Verma module}, and call the weight $\mu$ is a \emph{basic weight}. We can assume that $I=\Delta\backslash\{\alpha_i\}$ and $K=\Delta\backslash\{\alpha_k\}$ for some $i, k\in\{1, \ldots, n\}$. Thus the category $\caO_\lambda^{\frp}$ with $\frp=\frp_I$ is determined by the triple $(\Phi, i, k)$ which we called a \emph{basic system}. One of the main result in \cite{XZ} is the following classification of basic systems.

\begin{theorem}[{\cite[Theorem 5.8]{XZ}}]\label{jchm1}
A basic system $(\Phi, i, k)$ must be one of the following cases.
\begin{itemize}
\item [(1)] $(A_1, 1, 1)$, $(A_2, 1, 1)$, $(A_2, 1, 2)$, $(A_2, 2, 1)$, $(A_2, 2, 2)$, $(A_3, 2, 2)$;

\item [(2)] $(B_2, 1, 1)$, $(B_2, 1, 2)$, $(B_2, 2, 1)$, $(B_2, 2, 2)$, $(B_3, 2, 2)$, $(B_3, 2, 3)$, $(B_3, 3, 2)$, $(B_4, 3, 3)$;

\item [(3)] $(C_2, 1, 1)$, $(C_2, 1, 2)$, $(C_2, 2, 1)$, $(C_2, 2, 2)$, $(C_3, 2, 2)$,  $(C_3, 2, 3)$, $(C_3, 3, 2)$, $(C_4, 3, 3)$;

\item [(4)] $(D_4, 2, 2)$, $(D_5, 3, 3)$;

\item [(5)] $(E_6, 4, 4)$, $(E_7, 4, 4)$, $(E_7, 4, 5)$, $(E_7, 5, 4)$, $(E_8, 3, 4)$, $(E_8, 4, 3)$, $(E_8, 4, 4)$, $(E_8, 4, 5)$, $(E_8, 5, 4)$, $(E_8, 5, 5)$;

\item [(6)] $(F_4, 2, 2)$, $(F_4, 2, 3)$, $(F_4, 3, 2)$, $(F_4, 3, 3)$;

\item [(7)] $(G_2, 1, 1)$, $(G_2, 1, 2)$, $(G_2, 2, 1)$, $(G_2, 2, 2)$.
\end{itemize}
\end{theorem}

When the category $\caO_\lambda^{\frp}$ associated with $(\Phi, i, k)$ is semisimple, the Jantzen coefficients of the corresponding basic generalized Verma modules are zero. It suffices to consider the non semisimple cases \cite[\S6]{XZ}: $(A_1, 1, 1)$, $(B_3, 2, 2)$ and $(C_3, 2, 2)$, $(E_7, 4, 4)$, $(E_8, 4, 5)$, $(E_8, 5, 4)$ and $(E_8, 4, 4)$. Their basic weights and Jantzen coefficients are given in \cite{XZ}. By \cite[Lemma 5.2]{XZ}, all the basic weights of the basic system $(\Phi, i, k)$ are of the form $cw\varpi_k$ for some $c\in\mathbb{Z}^{>0}$ and $w\in {}^IW^K$. Note that $cw\varpi_k$ and $w\varpi_k$ lie in the same facet, it follows from \cite[Theorem 7.8]{H3} and the fact that Zuckerman functors commute with the translation function functors that there is an equivalence of categories which sends $M_I(cw\varpi_k)$ to $M_I(w\varpi_k)$ so that their radical filtration are in bijective correspondence. Therefore, it suffices to determine the radical filtration of those $M_I(w\varpi_k)$ for $w\in {}^IW^J$. For convenience, the basic weights are parameterized as $\lambda^j=x_j\varpi_k$, where $x_j\in{}^IW^K$ for $1\leq j\leq l$. Here we adopt the ordering in \cite[\S5]{XZ} which makes $s<t$ whenever $\lambda^s>\lambda^t$. We also have $\lambda^j=w_Iy_j\lambda$ with $y_j=w_Ix_jw_0\in {}^IW^{J}$ and $\lambda=w_0\varpi_k$ in view of Proposition \ref{bnprop1}. Moreover, one has $s<t$ and $x_s<x_t$ whenever $y_s>y_t$.

If the category $\caO_\lambda^\frp$ associated with $(\Phi, i, k)$ is not semisimple, then $\Phi$ is one of $A_1$, $B_3$, $C_3$, $E_7$ and $E_8$. It follows that $-w_0$ fix every simple root and $K=J$. For convenience, denote $M_s=M_I(\lambda^s)$, $L_s=L(\lambda^s)$, $c_{s, t}=c(\lambda^s, \lambda^t)$, ${}^IP^J_{s, t}={}^IP^J_{y_s, y_t}$ and ${}^IQ^J_{s, t}={}^IQ^J_{y_s, y_t}$.


\subsection{Radical filtrations associated with $(A_1, 1, 1)$, $(B_3, 2, 2)$ and $(C_3, 2, 2)$} These three basic systems share the same number of basic weights. The category $\caO_\lambda^\frp$ contains two generalized Verma modules. The unique nonzero Jantzen coefficient is $c_{1, 2}=1$ (see \cite[\S6]{XZ}). In view of (\ref{jceq2}), we have $\sum_{i\geq 1}i[\Rad_{i}M_2]=0$. This forces $\Rad_{i}M_2=0$ for $i\geq1$, that is, $M_2=L_2$ is a simple module. Similarly, we obtain
\[
\sum_{i\geq 1}i[\Rad_{i}M_1]=[M_2]=[L_2].
\]
One must have $\Rad_{i}M_1=0$ for $i>1$ and $\Rad_{1}M_1=L_2$. The radical filtration of $M_1$ is $\begin{tabular}{p{0.3cm}<{\centering}}
 $L_1$ \\
 $L_2$ \\
\end{tabular}$. The $\Ext^1$ poset of the categories is given in Figure \ref{jcfg1}.

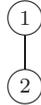
\begin{figure}[htbp]\footnotesize
\setlength{\unitlength}{0.9mm}
\begin{center}
\begin{picture}(0,10)
\put(0,0){\circle{5}}
\put(0,10){\circle{5}}

\put(0,2.5){\line(0,1){5}}

\put(-0.8,-1){$2$}
\put(-0.8,9){$1$}
\end{picture}
\end{center}
\caption{$\Ext^1$ posets for $(A_1, 1, 1)$, $(B_3, 2, 2)$ and $(C_3, 2, 2)$}
\label{jcfg1}
\end{figure}

With the classification of basic weights in \cite{XZ}, one has $\ell(y_1)=1$ and $\ell(y_2)=0$. Therefore ${}^IQ^J_{1, 2}=1={}^IP^J_{2,1}$ by (\ref{ikleq3}) and Corollary \ref{iklcor1}.

Evidently, $\lambda^2$ is the unique socular weight and $m(\lambda^2)=\lambda^1$ (see Lemma \ref{gklem4}).

\subsection{Radical filtrations associated with $(E_7, 4, 4)$} Now the category $\caO_\lambda^{\frp}$ contains $6$ generalized Verma modules. All the nonzero Jantzen coefficients are given in Table \ref{jctb1} (see \cite[\S6]{XZ}).

\begin{table}[htbp]
\begin{tabular}{|c|c|c|c|}
\hline
$i$ & $\{j\mid c_{i, j}=1\}$ & $\{j\mid c_{i, j}=-1\}$ & $\{j\mid c_{i, j}=2\}$ \\
\hline
$1$ & $3$ & $5$ & $6$ \\
\hline
$2$ & $3$ & $6$ & $5$ \\
\hline
$3$ &  &  & $4$ \\
\hline
$4$ & $5, 6$ &  & \\
\hline
\end{tabular}
\bigskip
\caption{Nonzero Jantzen coefficients of $(E_7, 4, 4)$}
\label{jctb1}
\end{table}

First (\ref{jceq2}) implies that $M_6=L_6$ and $M_5=L_5$. Next consider $M_4$. One has
\[
\sum_{i\geq 1}i[\Rad_{i}M_4]=c(4,5)[M_5]+c(4,6)[M_6]=[L_5]+[L_6].
\]
This forces $\Rad_{1}M_4=L_5\oplus L_6$ and $\Rad_{i}M_4=0$ for $i>1$. Then consider $M_3$.
\[
\sum_{i\geq 1}i[\Rad_{i}M_3]=c(3,4)[M_4]=2[L_4]+2[L_5]+2[L_6].
\]
We need to apply the parity property of radical filtrations (see Lemma \ref{ikllem2}) of generalized Verma modules. The length of $y_i$ can be calculated from the weights $\lambda^i$. They are given in Table \ref{jctb2}.

\begin{table}[htbp]
\begin{tabular}{|c|c|c|c|c|c|c|}
\hline
$i$ & $1$ & $2$ & $3$ & $4$ & $5$ & $6$ \\
\hline
$\ell(y_i)$ & $32$ & $26$ & $25$ & $18$ & $17$ & $11$\\
\hline
\end{tabular}
\bigskip
\caption{}
\label{jctb2}
\end{table}

The parity property show that $L_4$ must stay in the odd layers and $L_5, L_6$ must stay in the even layers of $M_3$. This forces $\Rad_{1}M_3=L_4\oplus L_4$, $\Rad_{2}M_3=L_5\oplus L_6$ and $\Rad_{i}M_3=0$ for $i>2$. More effort are needed to deal with $M_2$ and $M_1$. With
\[
\sum_{i\geq 1}i[\Rad_{i}M_2]=[M_3]+2[M_5]-[M_6]=[L_3]+2[L_4]+3[L_5],
\]
the parity property shows that $\Rad_2M_2=L_4$ and $L_3$ is direct summand of $\Rad_1M_2$. There are two possibilities for the position of $L_5$. Either the first layer $\Rad_1M_2$ contain three copies of $L_5$, or $\Rad_3M_2$ contains one copy of $L_5$. We are in a position to invoke the tool of Gelfand-Kirillov dimension. The basic weights $\lambda^1, \ldots, \lambda^6$ are presented in Table \ref{jctb3} (see \cite[\S5]{XZ}).

\renewcommand\arraystretch{1.3}

\begin{table}[H]
\begin{tabular}{|l|c|l|c|}
\hline
$i$ & $\lambda^i$ & $i$ & $\lambda^i$\\
\hline
$1$ & $(\frac{1}{2}, \frac{3}{2}, -\frac{3}{2}, -\frac{1}{2}, \frac{1}{2}, \frac{3}{2}, -\frac{3}{2}, \frac{3}{2})$ & $4$ & $(\frac{1}{2}, \frac{3}{2}, -\frac{5}{2}, -\frac{3}{2}, -\frac{1}{2}, \frac{1}{2}, -\frac{1}{2}, \frac{1}{2})$ \\
\hline
$2$ & $(0, 1, -2, -1, 0, 2, -1, 1)$ & $5$ & $(0, 1, -3, -1, 0, 1, 0, 0)$ \\
\hline
$3$ & $(0, 2, -2, -1, 0, 1, -1, 1)$ & $6$ & $(\frac{1}{2}, \frac{3}{2}, -\frac{5}{2}, -\frac{3}{2}, -\frac{1}{2}, \frac{1}{2}, \frac{1}{2}, -\frac{1}{2})$ \\
\hline
\end{tabular}
\bigskip
\caption{Basic weights of $(E_7, 4, 4)$}
\label{jctb3}
\end{table}

Using Lemma \ref{gklem3}, we can get $d=53$. Weyl's dimension formula \cite{H3} yields $c_d(M_1)=c_d(M_6)=2$, $c_d(M_2)=c_d(M_5)=4$ and $c_d(M_3)=c_d(M_4)=6$. The previous argument and Lemma \ref{gklem2} imply $c_d(L_6)=2$, $c_d(L_5)=4$, $c_d(L_3)=c_d(L_4)=0$ and
\[
c_d(\Rad_1M_2)\leq c_d(M_2)=4<3\times4=3c_d(L_5).
\]
So $\Rad_1M_2$ can not contain three copies of $L_5$. This means $\Rad_1M_2=L_3$ and $\Rad_3M_2=L_5$. Similarly, we can obtain the radical filtration of $M_1$. To summarize:

\[
M_1=\begin{aligned}
&L_1\\
&L_3\\
&L_4\\
&L_6
\end{aligned} \qquad
M_2=\begin{aligned}
&L_2\\
&L_3\\
&L_4\\
&L_5
\end{aligned} \qquad
M_3=\begin{aligned}
&L_3\\
&2L_4\\
L&_5L_6
\end{aligned} \qquad
M_4=\begin{aligned}
&L_4\\
L&_5L_6
\end{aligned} \qquad
M_5=\begin{aligned}
&L_5
\end{aligned} \
M_6=\begin{aligned}
&L_6
\end{aligned}
\]
Here $2L_4$ stands for $L_4\oplus L_4$ and $L_5L_6$ for $L_5\oplus L_6$. Evidently, $\lambda^5$ and $\lambda^6$ are socular weights, while $m(\lambda^5)=\lambda^2$ and $m(\lambda^6)=\lambda^1$.

\begin{figure}[htbp]\footnotesize
\setlength{\unitlength}{0.9mm}
\begin{center}
\begin{picture}(0,40)

\put(-10,5){\circle{5}} \put(10,5){\circle{5}}
\put(0,15){\circle{5}}
\put(0,25){\circle{5}}
\put(-10,35){\circle{5}} \put(10,35){\circle{5}}

\put(-8.15,6.85){\line(1,1){6.3}} \put(8.15,6.85){\line(-1,1){6.3}}
\put(0,17.5){\line(0,1){5}}
\put(1.85,26.85){\line(1,1){6.3}} \put(-1.85,26.85){\line(-1,1){6.3}}

\put(-10.8, 4){$5$}\put(9.2, 4){$6$}
\put(-0.8,14){$4$}
\put(-0.8,24){$3$}
\put(-10.8, 34){$1$}\put(9.2, 34){$2$}

\end{picture}
\end{center}
\caption{$\Ext^1$ poset for $(E_7, 4, 4)$}
\label{pbfig2}
\end{figure}

\subsection{Radical filtrations associated with $(E_8, 5, 4)$ and $(E_8, 4, 5)$} The argument for the case $(E_8, 5, 4)$ is relatively easy. The Jantzen coefficients \cite[\S6]{XZ} and parity property are enough for us to determine all the radical filtrations:

\bigskip

\[
\begin{aligned}
&L_{1}\\
L&_{2}L_{5}\\
&L_{3}
\end{aligned} \quad\
\begin{aligned}
&L_{2}\\
L_{3}&L_{4}L_{8}\\
&L_{6}
\end{aligned} \quad\
\begin{aligned}
&L_{3}\\
L&_{5}L_{6}\\
&L_{7}
\end{aligned} \quad\
\begin{aligned}
&L_{4}\\
&L_{6}\\
&L_{11}
\end{aligned} \quad\
\begin{aligned}
&L_{5}\\
&L_{7}\\
&L_{9}
\end{aligned} \quad\
\begin{aligned}
&L_{6}\\
L_{7}&L_{8}L_{11}\\
&L_{10}
\end{aligned} \quad\
\begin{aligned}
&L_{7}\\
L&_{9}L_{10}\\
&L_{12}
\end{aligned} \quad\
\begin{aligned}
&L_{8}\\
&L_{10}\\
&L_{15}
\end{aligned} \quad\
\begin{aligned}
&L_{9}\\
&L_{12}\\
&L_{14}
\end{aligned}
\]

\bigskip

\[
\begin{aligned}
&L_{10}\\
L_{11}&L_{12}L_{15}\\
&L_{13}
\end{aligned}\quad\
\begin{aligned}
&L_{11}\\
&L_{13}
\end{aligned} \quad\
\begin{aligned}
&L_{12}\\
L&_{13}L_{14}\\
&L_{16}
\end{aligned} \quad\
\begin{aligned}
&L_{13}\\
L&_{15}L_{16}\\
&L_{17}
\end{aligned} \quad\
\begin{aligned}
&L_{14}\\
&L_{16}
\end{aligned} \quad\
\begin{aligned}
&L_{15}\\
&L_{17}
\end{aligned} \quad\
\begin{aligned}
&L_{16}\\
&L_{17}\\
&L_{18}
\end{aligned} \quad\
\begin{aligned}
&L_{17}\\
&L_{18}
\end{aligned} \quad\
\begin{aligned}
&L_{18}
\end{aligned}
\]

\bigskip
There are many socular weights in this case. They are described in Table \ref{jctb4}.
\begin{table}[htbp]\footnotesize
\begin{tabular}{|c|c|c|c|c|c|c|c|c|c|c|c|c|c|}
\hline
Socular weight $\lambda^i$ & $\lambda^3$ & $\lambda^6$ & $\lambda^7$ & $\lambda^9$ & $\lambda^{10}$ & $\lambda^{11}$ & $\lambda^{12}$ & $\lambda^{13}$ & $\lambda^{14}$ & $\lambda^{15}$ & $\lambda^{16}$ & $\lambda^{17}$ & $\lambda^{18}$ \\
\hline
$m(\lambda^i)$ & $\lambda^1$ & $\lambda^2$ & $\lambda^3$ & $\lambda^5$ & $\lambda^{6}$ & $\lambda^{4}$ & $\lambda^{7}$ & $\lambda^{10}$ & $\lambda^{9}$ & $\lambda^{8}$ & $\lambda^{12}$ & $\lambda^{13}$ & $\lambda^{16}$ \\
\hline
\end{tabular}
\bigskip
\caption{}
\label{jctb4}
\end{table}

For the dual case $(E_8, 4, 5)$, we have to use the powerful tools of generalized Kazhdan-Lusztig polynomials. Note that $I=\Delta\backslash\{\alpha_4\}$ and $J=\Delta\backslash\{\alpha_5\}$. Define the map $g$ on the set $\{1, \ldots, 18\}$ such that $g(9)=9$, $g(10)=10$ and $g(i)=19-i$ otherwise. Corollary \ref{iklcor1} yields ${}^IP^J_{i, j}={}^JQ^I_{g(i), g(j)}$ (as in \S5.3, we need the basic weights described in \cite[\S5]{XZ} to get $y_i$). In view of Corollary \ref{gdcor1}, one has
\begin{equation}\label{45eq1}
\sum_{i\leq k\leq j}(-1)^{\ell(y_k)+\ell(y_i)}{}^JQ^I_{g(j), g(k)}{}^IQ^J_{i, k}=\delta_{i, j},
\end{equation}
The length function on $y_i$ is described in Table \ref{jctb5}. This is also the corresponding table for $(E_8, 5, 4)$.
\begin{table}[htbp]\footnotesize
\begin{tabular}{|c|c|c|c|c|c|c|c|c|c|c|c|c|c|c|c|c|c|c|}
\hline
$i$ & $1$ & $2$ & $3$ & $4$ & $5$ & $6$ & $7$ & $8$ & $9$ & $10$ & $11$ & $12$ & $13$ & $14$ & $15$ & $16$ & $17$ & $18$ \\
\hline
$\ell(y_i)$ & $70$ & $61$ & $60$ & $54$ & $53$ & $53$ & $50$ & $46$ & $45$ & $45$ & $44$ & $40$ & $37$ & $37$ & $36$ & $30$ & $29$ & $20$\\
\hline
\end{tabular}
\bigskip
\caption{}
\label{jctb5}
\end{table}
We can easily determine ${}^IQ^J_{i, j}$ from (\ref{ikleq3}), (\ref{45eq1}) and the radical filtrations associated with $(E_8, 5, 4)$. For example, if $i=1, j=2$, one has ${}^IQ^J_{1, 2}={}^JQ^I_{17, 18}=x^4$. If $i=2, j=3$, we have $Q_{2, 3}={}^JQ^I_{16, 17}=1$. Moreover,
\[
Q_{1, 3}={}^JQ^I_{17, 18}Q_{2, 3}-{}^JQ^I_{16, 18}=x^4-x^4=0.
\]
We can eventually get the full table of ${}^IQ^J_{i, j}$ in this fashion. With (\ref{ikleq3}) and Table \ref{jctb5}, the radical filtrations follows immediately:
\[
\begin{aligned}
&L_{1}\\
&L_{2}\\
&L_{4}\\
&L_{10}\\
&L_{12}\\
&L_{14}\\
&L_{18}
\end{aligned} \qquad
\begin{aligned}
&L_{2}\\
L&_{3}L_{4}\\
L_{5}&L_{6}L_{10}\\
L_{7}&L_{8}L_{12}\\
L_{10}&L_{13}L_{14}\\
L_{11}&L_{16}L_{18}\\
&L_{17}
\end{aligned} \qquad
\begin{aligned}
&L_{3}\\
L&_{5}L_{6}\\
L&_{7}L_{8}\\
L&_{10}L_{13}\\
L&_{11}L_{16}\\
&L_{17}
\end{aligned} \qquad
\begin{aligned}
&L_{4}\\
L&_{6}L_{10}\\
L_{7}&L_{8}L_{12}\\
L_{9}L&_{10}L_{13}L_{14}\\
L_{11}L&_{12}L_{16}L_{18}\\
L&_{13}L_{17}\\
&L_{15}
\end{aligned} \qquad
\begin{aligned}
&L_{5}\\
&L_{7}\\
&L_{10}\\
&L_{11}\\
&L_{17}\\
&L_{18}
\end{aligned} \qquad
\begin{aligned}
&L_{6}\\
L&_{7}L_{8}\\
L_{9}&L_{10}L_{13}\\
L_{11}&L_{12}L_{16}\\
L&_{13}L_{17}\\
&L_{15}
\end{aligned}
\]

\smallskip

\[
\begin{aligned}
&L_{7}\\
L&_{9}L_{10}\\
L&_{11}L_{12}\\
L&_{13}L_{17}\\
L&_{15}L_{18}
\end{aligned} \qquad\
\begin{aligned}
&L_{8}\\
L&_{10}L_{13}\\
L_{11}&L_{12}L_{16}\\
L_{13}&L_{14}L_{17}\\
L_{15}&L_{16}L_{18}\\
&L_{17}
\end{aligned} \qquad\
\begin{aligned}
&L_{9}\\
&L_{12}\\
&L_{13}\\
&L_{15}
\end{aligned} \qquad\
\begin{aligned}
&L_{10}\\
L&_{11}L_{12}\\
L_{13}&L_{14}L_{17}\\
L_{15}&L_{16}2L_{18}\\
&L_{17}
\end{aligned} \qquad\
\begin{aligned}
&L_{11}\\
L&_{13}L_{17}\\
L_{15}&L_{16}L_{18}\\
&L_{17}
\end{aligned}
\]

\bigskip

\[
\begin{aligned}
&L_{12}\\
L&_{13}L_{14}\\
L_{15}&L_{16}L_{18}\\
&L_{17}
\end{aligned} \qquad
\begin{aligned}
&L_{13}\\
L&_{15}L_{16}\\
&L_{17}
\end{aligned} \qquad
\begin{aligned}
&L_{14}\\
L&_{16}L_{18}\\
&L_{17}
\end{aligned} \qquad
\begin{aligned}
&L_{15}\\
&L_{17}\\
&L_{18}
\end{aligned} \qquad
\begin{aligned}
&L_{16}\\
&L_{17}
\end{aligned} \qquad
\begin{aligned}
&L_{17}\\
&L_{18}
\end{aligned} \qquad
\begin{aligned}
&L_{18}
\end{aligned}
\]
The socular weights are $\lambda^{15}$, $\lambda^{17}$ and $\lambda^{18}$, while $m(\lambda^{15})=\lambda^{4}$, $m(\lambda^{17})=\lambda^{2}$ and $m(\lambda^{18})=\lambda^1$.

\begin{figure}[H]
\setlength{\unitlength}{0.9mm}
\begin{center}
\begin{picture}(0,110)\footnotesize

\put(0,5){\circle{5}}
\put(0,15){\circle{5}}
\put(-10,25){\circle{5}} \put(10,25){\circle{5}}
\put(-10,35){\circle{5}} \put(10,35){\circle{5}}
\put(-10,45){\circle{5}} \put(10,45){\circle{5}}
\put(-10,55){\circle{5}} \put(10,55){\circle{5}}
\put(-10,65){\circle{5}} \put(10,65){\circle{5}}
\put(-10,75){\circle{5}} \put(10,75){\circle{5}}
\put(-10,85){\circle{5}} \put(10,85){\circle{5}}
\put(0,95){\circle{5}}
\put(0,105){\circle{5}}

\put(0,7.5){\line(0,1){5}}
\put(1.85,16.85){\line(1,1){6.3}} \put(-1.85,16.85){\line(-1,1){6.3}}
\put(-10,27.5){\line(0,1){5}} \put(-7.7,26.15){\line(2,1){15.4}}\put(10,27.5){\line(0,1){5}}
\put(-10,37.5){\line(0,1){5}}\put(7.7,36.15){\line(-2,1){15.4}} \put(10,37.5){\line(0,1){5}}
\put(-10,47.5){\line(0,1){5}} \put(-7.7,46.15){\line(2,1){15.4}}\put(10,47.5){\line(0,1){5}}
\put(-10,57.5){\line(0,1){5}}\put(7.7,56.15){\line(-2,1){15.4}} \put(10,57.5){\line(0,1){5}}
\put(-10,67.5){\line(0,1){5}} \put(-7.7,66.15){\line(2,1){15.4}}\put(10,67.5){\line(0,1){5}}
\put(-10,77.5){\line(0,1){5}}\put(7.7,76.15){\line(-2,1){15.4}} \put(10,77.5){\line(0,1){5}}
\put(-8.15,86.85){\line(1,1){6.3}} \put(8.15,86.85){\line(-1,1){6.3}}
\put(0,97.5){\line(0,1){5}}

\put(-1.5,4.0){$18$}
\put(-1.7,14){$17$}
\put(-11.7,24){$16$}\put(8.3,24){$15$}
\put(-11.7,34){$14$}\put(8.3,34){$13$}
\put(-11.7,44){$12$}\put(8.3,44){$11$}
\put(-10.8,54){$9$}\put(8.3,54){$10$}
\put(-10.8,64){$7$}\put(9.2,64){$8$}
\put(-10.8,74){$5$}\put(9.2,74){$6$}
\put(-10.8,84){$3$}\put(9.2,84){$4$}
\put(-0.6,94.0){$2$}
\put(-0.8,104){$1$}

\end{picture}
\end{center}
\caption{$\Ext^1$ poset for $(E_8, 4, 5)$ and $(E_8, 5, 4)$}
\label{jcfg3}
\end{figure}
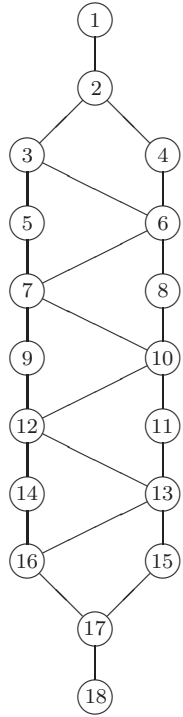

\subsection{Radical filtrations associated with $(E_8, 4, 4)$} The category $\caO_\lambda^\frp$ contains $47$ generalized Verma modules. See the length $\ell(y_i)$ in Table \ref{jctb6}.
\begin{table}[htbp]\footnotesize
\begin{tabular}{|c|c|c|c|c|c|c|c|c|c|c|c|c|c|c|c|c|c|c|}
\hline
$i$ & $1$ & $2$ & $3$ & $4$ & $5$ & $6$ & $7$ & $8$ & $9$ & $10$ & $11$ & $12$ & $13$ & $14$ & $15$ & $16$\\
\hline
$\ell(y_i)$ & $74$ & $67$ & $66$ & $66$ & $60$ & $59$ & $59$ & $58$ & $57$ & $57$ & $56$ & $52$ & $52$ & $52$ & $51$ & $51$\\
\hline
$i$ & $17$ & $18$ & $19$ & $20$ & $21$ & $22$ & $23$ & $24$ & $25$ & $26$ & $27$ & $28$ & $29$ & $30$ & $31$ & $32$ \\
\hline
$\ell(y_i)$ & $50$ & $50$ & $50$ & $49$ & $49$ & $46$ & $46$ & $46$ & $46$ & $46$ & $43$ & $43$ & $42$ & $42$ & $42$ & $41$ \\
\hline
$i$ & $33$ & $34$ & $35$ & $36$ & $37$ & $38$ & $39$ & $40$ & $41$ & $42$ & $43$ & $44$ & $45$ & $46$ & $47$ &  \\
\hline
$\ell(y_i)$ & $41$ & $40$ & $40$ & $40$ & $36$ & $35$ & $35$ & $34$ & $33$ & $33$ & $32$ & $26$ & $26$ & $25$ & $18$ &  \\
\hline
\end{tabular}
\bigskip
\caption{}
\label{jctb6}
\end{table}
One can determine the radical filtrations for most generalized Verma modules (from $M_{47}$ to $M_{18}$, including $M_{16}$ and $M_{14}$) by the previous methods although the process is quite time consuming. We might save some time if computer programs are used to take care of some tedious calculations. In order to further simplify the argument and finish the full table, we have to apply the powerful algorithm of duCloux \cite{Duc1} and the latest version of his computer program {\tt Coxeter}3 \cite{Duc2}. For the large Weyl group of $E_8$, it is not reasonable to obtain the Kazhdan-Lusztig polynomials directly. Even the leading coefficients of Kazhdan-Lusztig polynomials (the $\mu$-function \cite{KL}) are very difficult to calculate when $\ell(y_i)$ is large. Fortunately those $\mu$-functions which can be obtained by the program {\tt Coxeter}3 in an ordinary PC, provide enough information of the radical filtrations:

\[
\begin{aligned}
&L_{1}\\
2L_{2}&L_{6}L_{9}\\
L_{3}L_{4}&L_{5}L_{13}L_{14}\\
L&_{7}L_{10}\\
L&_{17}L_{18}\\
L&_{28}L_{33}\\
L_{36}&L_{37}L_{40}\\
&L_{39}\\
&L_{47}
\end{aligned} \qquad
\begin{aligned}
&L_{2}\\
L_{3}L_{4}&L_{5}L_{13}L_{14}\\
L_{6}2L&_{7}L_{9}2L_{10}\\
L_{8}L_{11}L_{13}&L_{14}2L_{17}2L_{18}L_{22}\\
L_{15}L_{16}&L_{20}2L_{28}2L_{33}\\
L_{24}L_{25}L_{30}&L_{31}2L_{36}2L_{37}2L_{40}\\
L_{32}L_{33}&L_{38}2L_{39}L_{41}\\
L_{35}L&_{44}L_{45}2L_{47}\\
&L_{46}
\end{aligned}\qquad
\begin{aligned}
&L_{3}\\
L_{6}L&_{7}L_{9}L_{10}\\
L_{8}L_{11}L_{13}&L_{14}L_{17}L_{18}L_{22}\\
L_{15}L_{16}&L_{20}L_{28}L_{33}\\
L_{24}L_{25}L_{30}&L_{31}L_{36}L_{37}L_{40}\\
L_{32}L_{33}&L_{38}L_{39}L_{41}\\
L_{35}L&_{44}L_{45}L_{47}\\
&L_{46}
\end{aligned}
\]

\bigskip

\[
\begin{aligned}
&L_{4}\\
L_{6}L&_{7}L_{9}L_{10}\\
L_{8}L_{11}L_{13}&L_{14}L_{17}L_{18}L_{22}\\
L_{15}L_{16}&L_{20}L_{28}L_{33}\\
L_{24}L_{25}L_{30}&L_{31}L_{36}L_{37}L_{40}\\
L_{32}L_{33}&L_{38}L_{39}L_{41}\\
L_{35}L&_{44}L_{45}L_{47}\\
&L_{46}
\end{aligned}\qquad
\begin{aligned}
&L_{5}\\
L&_{6}L_{7}\\
L_{8}L&_{17}L_{18}L_{22}\\
L_{15}L&_{20}L_{28}L_{33}\\
L_{22}L_{24}L_{25}L&_{30}L_{31}L_{36}L_{37}L_{40}\\
L_{27}L_{32}L&_{33}L_{38}L_{39}L_{41}\\
L_{35}L_{37}&L_{44}L_{45}L_{47}\\
L&_{41}L_{46}\\
&L_{43}
\end{aligned} \qquad
\begin{aligned}
&L_{6}\\
L&_{8}L_{22}\\
L&_{15}L_{20}\\
L_{22}L_{24}&L_{25}L_{30}L_{31}\\
L_{27}L_{32}&L_{33}L_{38}L_{41}\\
L_{35}L&_{37}L_{44}L_{45}\\
L&_{41}L_{46}\\
&L_{43}
\end{aligned}
\]

\bigskip

\[
\begin{aligned}
&L_{7}\\
L_{8}L_{11}L_{13}&L_{14}L_{17}L_{18}L_{22}\\
L_{9}L_{10}2L_{15}&L_{16}2L_{20}L_{28}L_{33}\\
L_{11}L_{12}L_{13}L_{17}L_{18}L_{22}&2L_{24}2L_{25}2L_{30}2L_{31}L_{36}L_{37}L_{40}\\
L_{15}L_{16}L_{20}L_{21}L_{27}&L_{28}2L_{32}3L_{33}2L_{38}L_{39}2L_{41}\\
L_{18}L_{19}L_{23}L_{25}L_{30}L_{31}&2L_{35}L_{36}2L_{37}L_{40}2L_{44}2L_{45}L_{47}\\
L_{21}L_{28}L_{33}&L_{38}L_{39}2L_{41}2L_{46}\\
L_{26}L_{40}L_{43}&L_{44}L_{45}L_{47}\\
&L_{42}
\end{aligned}
\]

\bigskip

\[
\begin{aligned}
&L_{8}\\
L_{9}L&_{10}L_{15}L_{20}\\
L_{11}L_{12}L_{13}L_{17}L&_{18}L_{22}L_{24}L_{25}L_{30}L_{31}\\
L_{15}L_{16}L_{20}L_{21}L&_{27}L_{28}L_{32}2L_{33}L_{38}L_{41}\\
L_{18}L_{19}L_{23}L_{25}L_{30}L&_{31}L_{35}L_{36}2L_{37}L_{40}L_{44}L_{45}\\
L_{21}L_{28}L_{33}&L_{38}L_{39}2L_{41}L_{46}\\
L_{26}L_{40}L&_{43}L_{44}L_{45}L_{47}\\
&L_{42}
\end{aligned}\qquad
\begin{aligned}
&L_{9}\\
L_{11}&L_{12}L_{13}\\
L_{15}L&_{16}L_{20}L_{21}\\
L_{18}L_{19}L&_{23}L_{25}L_{30}L_{31}\\
L_{21}L_{28}&L_{33}L_{38}L_{41}\\
L_{26}L&_{40}L_{44}L_{45}\\
&L_{42}
\end{aligned}
\]

\bigskip

\[
\begin{aligned}
&L_{10}\\
L_{11}L_{13}&L_{14}L_{17}L_{18}\\
L_{15}2L_{16}L&_{20}L_{21}L_{28}L_{33}\\
L_{18}L_{19}L_{23}L_{24}2L&_{25}L_{30}L_{31}L_{36}L_{37}L_{40}\\
L_{21}L_{28}L_{32}&2L_{33}L_{38}L_{39}L_{41}\\
L_{26}L_{35}L&_{40}L_{44}L_{45}L_{47}\\
L&_{42}L_{46}\\
&L_{47}
\end{aligned} \qquad
\begin{aligned}
&L_{11}\\
L_{15}L&_{16}L_{20}L_{21}\\
L_{18}L_{19}L_{22}L&_{23}L_{24}2L_{25}L_{30}L_{31}\\
L_{21}L_{27}L_{28}&L_{32}2L_{33}L_{38}L_{41}\\
L_{26}L_{35}L&_{37}L_{40}L_{44}L_{45}\\
L_{41}&L_{42}L_{46}\\
L&_{43}L_{47}
\end{aligned}
\]

\bigskip

\[
\begin{aligned}
&L_{12}\\
L&_{15}L_{20}\\
L_{17}L_{18}L_{22}&L_{23}L_{25}L_{30}L_{31}\\
L_{20}L_{21}L_{27}&2L_{28}2L_{33}L_{38}L_{41}\\
L_{23}L_{25}L_{26}L_{30}L&_{31}L_{36}2L_{37}2L_{40}L_{44}L_{45}\\
L_{27}L_{28}L_{32}L&_{33}L_{38}L_{39}2L_{41}L_{42}\\
L_{29}L_{30}L_{35}L_{37}&L_{40}L_{43}L_{44}L_{45}L_{47}\\
L_{32}L&_{38}L_{41}L_{46}\\
&L_{34}
\end{aligned} \qquad
\begin{aligned}
&L_{13}\\
L&_{15}L_{16}\\
L_{18}L_{19}L&_{24}L_{25}L_{30}L_{31}\\
L_{21}L_{28}L&_{32}2L_{33}L_{38}L_{41}\\
L_{26}L_{35}L_{36}&L_{37}L_{40}L_{44}L_{45}\\
L_{39}L&_{41}L_{42}L_{46}\\
L_{43}L&_{44}L_{45}2L_{47}\\
&L_{46}
\end{aligned}
\]

\bigskip

\[
\begin{aligned}
&L_{14}\\
&L_{16}\\
L&_{24}L_{25}\\
L&_{32}L_{33}\\
&L_{35}\\
&L_{46}\\
&L_{47}
\end{aligned} \qquad
\begin{aligned}
&L_{15}\\
L_{17}L_{18}L_{22}&L_{24}L_{25}L_{30}L_{31}\\
L_{20}L_{21}L_{27}2L&_{28}L_{32}3L_{33}L_{38}L_{41}\\
L_{23}L_{25}L_{26}L_{30}L_{31}&L_{35}2L_{36}3L_{37}2L_{40}L_{44}L_{45}\\
L_{27}L_{28}L_{32}L_{33}&L_{38}2L_{39}3L_{41}L_{42}L_{46}\\
L_{29}L_{30}L_{35}L_{37}&L_{40}2L_{43}2L_{44}2L_{45}3L_{47}\\
L_{32}L&_{38}L_{41}2L_{46}\\
&L_{34}
\end{aligned} \qquad
\begin{aligned}
&L_{16}\\
L_{18}L&_{19}L_{24}L_{25}\\
L_{21}L&_{28}L_{32}2L_{33}\\
L_{26}L_{35}&L_{36}L_{37}L_{40}\\
L_{39}L&_{41}L_{42}L_{46}\\
L_{43}L&_{44}L_{45}3L_{47}\\
&L_{46}
\end{aligned}
\]

\bigskip

\[
\begin{aligned}
&L_{17}\\
L_{20}&L_{28}L_{33}\\
L_{23}L_{25}L_{30}&L_{31}L_{36}L_{37}L_{40}\\
L_{27}L_{28}L_{32}&L_{33}L_{38}L_{39}L_{41}\\
L_{29}L_{30}L_{35}L&_{37}L_{40}L_{44}L_{45}L_{47}\\
L_{32}L&_{38}L_{41}L_{46}\\
&L_{34}
\end{aligned} \qquad
\begin{aligned}
&L_{18}\\
L_{20}L&_{21}L_{28}L_{33}\\
L_{23}L_{25}L_{26}L&_{30}L_{31}L_{36}L_{37}L_{40}\\
L_{27}L_{28}L_{32}L&_{33}L_{38}L_{39}2L_{41}L_{42}\\
L_{29}L_{30}L_{35}L_{37}&L_{40}L_{43}2L_{44}2L_{45}2L_{47}\\
L_{32}L&_{38}L_{41}2L_{46}\\
&L_{34}
\end{aligned}
\]

\bigskip

\[
\begin{aligned}
&L_{19}\\
&L_{21}\\
&L_{26}\\
L&_{41}L_{42}\\
L_{43}L&_{44}L_{45}L_{47}\\
&L_{46}
\end{aligned} \quad
\begin{aligned}
&L_{20}\\
L_{22}L_{23}&L_{25}L_{30}L_{31}\\
2L_{27}L_{28}L&_{32}L_{33}L_{38}L_{41}\\
L_{29}L_{30}L_{35}&2L_{37}L_{40}L_{44}L_{45}\\
L_{32}L&_{38}2L_{41}L_{46}\\
L&_{34}L_{43}
\end{aligned} \qquad
\begin{aligned}
&L_{21}\\
L_{23}&L_{25}L_{26}\\
L_{27}L_{28}L&_{32}L_{33}L_{41}L_{42}\\
L_{29}L_{30}L_{35}L_{37}&L_{40}L_{43}L_{44}L_{45}L_{47}\\
L_{32}L&_{38}L_{41}2L_{46}\\
L&_{34}L_{47}
\end{aligned}
\]

\bigskip

\[
\begin{aligned}
&L_{22}\\
&L_{27}\\
&L_{37}\\
&L_{41}\\
&L_{43}
\end{aligned} \qquad
\begin{aligned}
&L_{23}\\
L&_{27}L_{28}\\
L_{29}L&_{30}L_{37}L_{40}\\
L_{32}&L_{38}L_{41}\\
&L_{34}
\end{aligned} \qquad
\begin{aligned}
&L_{24}\\
L&_{32}L_{33}\\
L_{35}&L_{36}L_{37}\\
L_{39}&L_{41}L_{46}\\
L_{43}L&_{44}L_{45}2L_{47}\\
&L_{46}
\end{aligned} \qquad
\begin{aligned}
&L_{25}\\
L_{27}L&_{28}L_{32}L_{33}\\
L_{29}L_{30}L&_{35}L_{36}2L_{37}L_{40}\\
L_{32}L_{38}&L_{39}2L_{41}L_{46}\\
L_{34}L_{43}&L_{44}L_{45}2L_{47}\\
&L_{46}
\end{aligned}
\]

\bigskip

\[
\begin{aligned}
&L_{26}\\
L_{28}L&_{33}L_{41}L_{42}\\
L_{30}L_{31}L_{35}L_{36}L&_{37}L_{40}L_{43}L_{44}L_{45}L_{47}\\
L_{32}L_{33}2L&_{38}L_{39}2L_{41}2L_{46}\\
L_{34}L_{35}L_{37}&L_{40}2L_{44}2L_{45}2L_{47}\\
L_{38}L&_{39}L_{41}2L_{46}\\
L_{40}L_{43}&L_{44}L_{45}L_{47}\\
&L_{42}
\end{aligned} \quad
\begin{aligned}
&L_{27}\\
L_{29}&L_{30}L_{37}\\
L_{32}&L_{38}2L_{41}\\
L_{34}L&_{43}L_{44}L_{45}\\
&L_{46}
\end{aligned} \quad
\begin{aligned}
&L_{28}\\
L_{30}L_{31}&L_{36}L_{37}L_{40}\\
L_{32}L_{33}&2L_{38}L_{39}2L_{41}\\
L_{34}L_{35}L_{37}&L_{40}2L_{44}2L_{45}L_{47}\\
L_{38}L&_{39}L_{41}2L_{46}\\
L_{40}L_{43}&L_{44}L_{45}L_{47}\\
&L_{42}
\end{aligned}
\]

\bigskip

\[
\begin{aligned}
&L_{29}\\
&L_{32}\\
&L_{34}\\
&L_{46}\\
&L_{47}
\end{aligned} \qquad
\begin{aligned}
&L_{30}\\
L_{32}L&_{33}L_{38}L_{41}\\
L_{34}L_{35}L&_{37}L_{40}L_{44}L_{45}\\
L_{38}L&_{39}L_{41}2L_{46}\\
L_{40}L_{43}&L_{44}L_{45}2L_{47}\\
&L_{42}
\end{aligned} \qquad
\begin{aligned}
&L_{31}\\
L_{33}&L_{38}L_{41}\\
L_{35}L_{37}&L_{40}L_{44}L_{45}\\
L_{38}L&_{39}L_{41}L_{46}\\
L_{40}L_{43}&L_{44}L_{45}L_{47}\\
&L_{42}
\end{aligned} \qquad
\begin{aligned}
&L_{32}\\
L_{34}&L_{35}L_{37}\\
L_{38}L&_{39}2L_{41}2L_{46}\\
L_{40}2L_{43}&2L_{44}2L_{45}3L_{47}\\
L&_{42}L_{46}
\end{aligned}
\]

\bigskip

\[
\begin{aligned}
&L_{33}\\
L_{35}L&_{36}L_{37}L_{40}\\
L_{38}2L&_{39}L_{41}L_{46}\\
L_{40}L_{43}&2L_{44}2L_{45}3L_{47}\\
L&_{42}L_{46}
\end{aligned} \quad\
\begin{aligned}
&L_{34}\\
L_{38}&L_{41}L_{46}\\
L_{40}L_{43}&L_{44}L_{45}L_{47}\\
&L_{42}
\end{aligned} \qquad
\begin{aligned}
&L_{35}\\
L_{38}L&_{39}L_{41}L_{46}\\
L_{40}L_{43}&2L_{44}2L_{45}2L_{47}\\
L&_{42}L_{46}
\end{aligned}\qquad
\begin{aligned}
&L_{36}\\
&L_{39}\\
L_{44}&L_{45}L_{47}\\
&L_{46}
\end{aligned}
\]

\bigskip

\[
\begin{aligned}
&L_{37}\\
L_{38}&L_{39}L_{41}\\
L_{40}L_{43}&2L_{44}2L_{45}L_{47}\\
L&_{42}L_{46}
\end{aligned} \qquad
\begin{aligned}
&L_{38}\\
L_{40}&L_{44}L_{45}\\
&L_{42}
\end{aligned} \qquad
\begin{aligned}
&L_{39}\\
L_{40}L&_{44}L_{45}L_{47}\\
L&_{42}L_{46}\\
&L_{47}
\end{aligned} \qquad
\begin{aligned}
&L_{40}\\
L&_{41}L_{42}\\
L_{43}L&_{44}L_{45}L_{47}\\
&L_{46}
\end{aligned}
\]

\bigskip

\[
\begin{aligned}
&L_{41}\\
L_{43}&L_{44}L_{45}\\
&L_{46}
\end{aligned}\qquad
\begin{aligned}
&L_{42}\\
L_{43}L&_{44}L_{45}L_{47}\\
&2L_{46}\\
&L_{47}
\end{aligned} \qquad
\begin{aligned}
&L_{43}\\
&L_{46}\\
&L_{47}
\end{aligned} \qquad
\begin{aligned}
&L_{44}\\
&L_{46}\\
&L_{47}
\end{aligned}\qquad
\begin{aligned}
&L_{45}\\
&L_{46}\\
&L_{47}
\end{aligned} \qquad
\begin{aligned}
&L_{46}\\
&2L_{47}
\end{aligned} \qquad
\begin{aligned}
&L_{47}
\end{aligned}
\]

\bigskip
There are five socular weights $\lambda^{34}$, $\lambda^{42}$, $\lambda^{43}$, $\lambda^{46}$ and $\lambda^{47}$, while $m(\lambda^{34})=\lambda^{12}$, $m(\lambda^{42})=\lambda^7$, $m(\lambda^{43})=\lambda^{5}$, $m(\lambda^{46})=\lambda^2$ and $m(\lambda^{47})=\lambda^1$.

\[
\begin{aligned}
\end{aligned} \qquad
\begin{aligned}
\end{aligned} \qquad
\begin{aligned}
\end{aligned}
\]

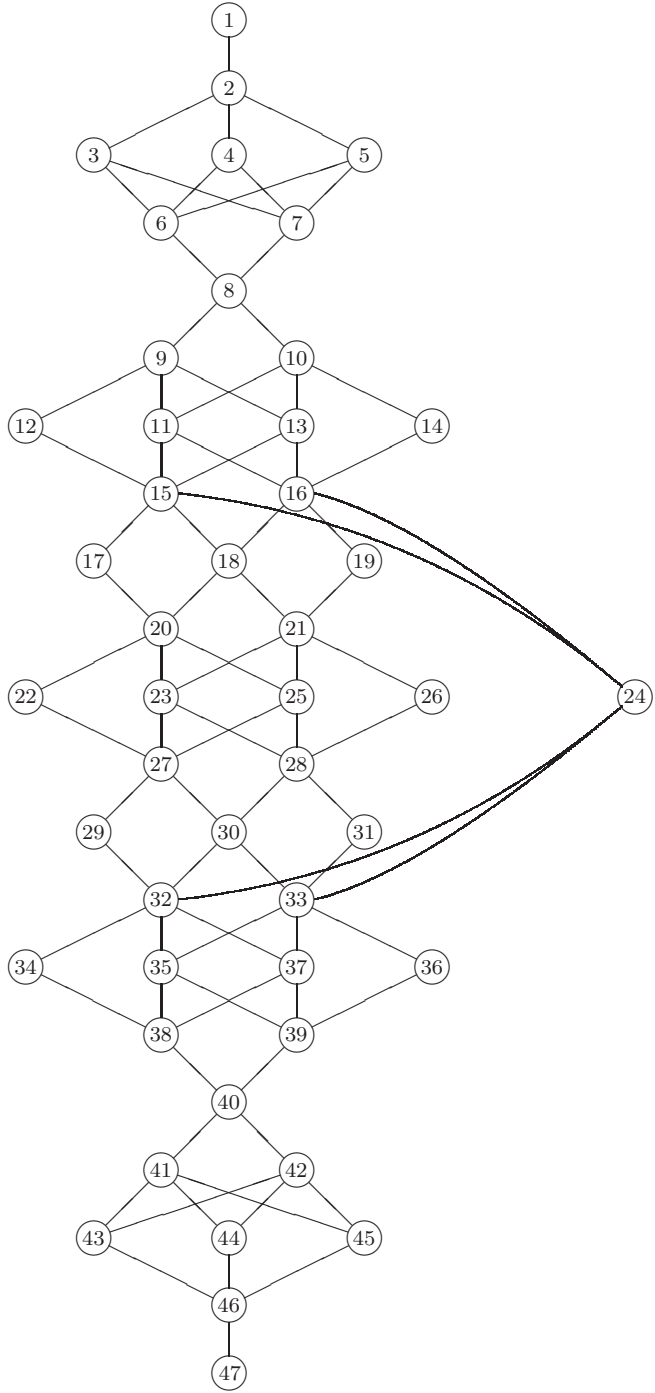
\begin{figure}[H]\footnotesize
\setlength{\unitlength}{0.9mm}
\begin{center}
\begin{picture}(0,210)

\put(0,5){\circle{5}}
\put(0,15){\circle{5}}
\put(-20,25){\circle{5}} \put(0,25){\circle{5}} \put(20,25){\circle{5}}
\put(-10,35){\circle{5}} \put(10,35){\circle{5}}
\put(0,45){\circle{5}}
\put(-10,55){\circle{5}} \put(10,55){\circle{5}}
\put(-30,65){\circle{5}}\put(-10,65){\circle{5}} \put(10,65){\circle{5}}\put(30,65){\circle{5}}
\put(-10,75){\circle{5}} \put(10,75){\circle{5}}
\put(-20,85){\circle{5}} \put(0,85){\circle{5}} \put(20,85){\circle{5}}
\put(-10,95){\circle{5}} \put(10,95){\circle{5}}
\put(-30,105){\circle{5}}\put(-10,105){\circle{5}} \put(10,105){\circle{5}}\put(30,105){\circle{5}}\put(60,105){\circle{5}}
\put(-10,115){\circle{5}} \put(10,115){\circle{5}}
\put(-20,125){\circle{5}} \put(0,125){\circle{5}} \put(20,125){\circle{5}}
\put(-10,135){\circle{5}} \put(10,135){\circle{5}}
\put(-30,145){\circle{5}}\put(-10,145){\circle{5}} \put(10,145){\circle{5}}\put(30,145){\circle{5}}
\put(-10,155){\circle{5}} \put(10,155){\circle{5}}
\put(0,165){\circle{5}}
\put(-10,175){\circle{5}} \put(10,175){\circle{5}}
\put(-20,185){\circle{5}} \put(0,185){\circle{5}} \put(20,185){\circle{5}}
\put(0,195){\circle{5}}
\put(0,205){\circle{5}}

\put(0,7.5){\line(0,1){5}}
\put(-2.3,16.15){\line(-2,1){15.4}} \put(0,17.5){\line(0,1){5}} \put(2.3,16.15){\line(2,1){15.4}}
\put(-18.15,26.85){\line(1,1){6.3}} \put(-17.54,25.82){\line(3,1){25.1}}
\put(1.85,26.85){\line(1,1){6.3}} \put(-1.85,26.85){\line(-1,1){6.3}}
\put(17.54,25.82){\line(-3,1){25.1}} \put(18.15,26.85){\line(-1,1){6.3}}
\put(-8.15,36.85){\line(1,1){6.3}} \put(8.15,36.85){\line(-1,1){6.3}}
\put(1.85,46.85){\line(1,1){6.3}} \put(-1.85,46.85){\line(-1,1){6.3}}
\put(-12.3,56.15){\line(-2,1){15.4}} \put(-10,57.5){\line(0,1){5}} \put(-7.7,56.15){\line(2,1){15.4}}
\put(7.7,56.15){\line(-2,1){15.4}} \put(10,57.5){\line(0,1){5}} \put(12.3,56.15){\line(2,1){15.4}}
\put(-27.7,66.15){\line(2,1){15.4}} \put(-10,67.5){\line(0,1){5}} \put(-7.7,66.15){\line(2,1){15.4}}
\put(7.7,66.15){\line(-2,1){15.4}} \put(10,67.5){\line(0,1){5}} \put(27.7,66.15){\line(-2,1){15.4}}
\put(-11.85,76.85){\line(-1,1){6.3}} \put(-8.15,76.85){\line(1,1){6.3}}
\put(8.15,76.85){\line(-1,1){6.3}} \put(11.85,76.85){\line(1,1){6.3}}
\put(-18.15,86.85){\line(1,1){6.3}}
\put(1.85,86.85){\line(1,1){6.3}} \put(-1.85,86.85){\line(-1,1){6.3}}
\put(18.15,86.85){\line(-1,1){6.3}}
\put(-12.3,96.15){\line(-2,1){15.4}} \put(-10,97.5){\line(0,1){5}} \put(-7.7,96.15){\line(2,1){15.4}}
\put(7.7,96.15){\line(-2,1){15.4}} \put(10,97.5){\line(0,1){5}} \put(12.3,96.15){\line(2,1){15.4}}
\put(-27.7,106.15){\line(2,1){15.4}} \put(-10,107.5){\line(0,1){5}} \put(-7.7,106.15){\line(2,1){15.4}}
\put(7.7,106.15){\line(-2,1){15.4}} \put(10,107.5){\line(0,1){5}} \put(27.7,106.15){\line(-2,1){15.4}}
\put(-11.85,116.85){\line(-1,1){6.3}} \put(-8.15,116.85){\line(1,1){6.3}}
\put(8.15,116.85){\line(-1,1){6.3}} \put(11.85,116.85){\line(1,1){6.3}}
\put(-18.15,126.85){\line(1,1){6.3}}
\put(1.85,126.85){\line(1,1){6.3}} \put(-1.85,126.85){\line(-1,1){6.3}}
\put(18.15,126.85){\line(-1,1){6.3}}
\put(-12.3,136.15){\line(-2,1){15.4}} \put(-10,137.5){\line(0,1){5}} \put(-7.7,136.15){\line(2,1){15.4}}
\put(7.7,136.15){\line(-2,1){15.4}} \put(10,137.5){\line(0,1){5}} \put(12.3,136.15){\line(2,1){15.4}}
\put(-27.7,146.15){\line(2,1){15.4}} \put(-10,147.5){\line(0,1){5}} \put(-7.7,146.15){\line(2,1){15.4}}
\put(7.7,146.15){\line(-2,1){15.4}} \put(10,147.5){\line(0,1){5}} \put(27.7,146.15){\line(-2,1){15.4}}
\put(-8.15,156.85){\line(1,1){6.3}} \put(8.15,156.85){\line(-1,1){6.3}}
\put(1.85,166.85){\line(1,1){6.3}} \put(-1.85,166.85){\line(-1,1){6.3}}
\put(-11.85,176.85){\line(-1,1){6.3}} \put(-8.15,176.85){\line(1,1){6.3}}\put(-7.54,175.82){\line(3,1){25.1}}
\put(8.15,176.85){\line(-1,1){6.3}} \put(11.85,176.85){\line(1,1){6.3}}\put(7.54,175.82){\line(-3,1){25.1}}
\put(17.7,186.15){\line(-2,1){15.4}} \put(0,187.5){\line(0,1){5}} \put(-17.7,186.15){\line(2,1){15.4}}
\put(0,197.5){\line(0,1){5}}

\qbezier(12.7,75)(29,79)(58,103.5)
\qbezier(-7.3,75)(29,79)(58,103.5)
\qbezier(12.7,135)(29,131)(58,106.5)
\qbezier(-7.3,135)(29,131)(58,106.5)

\put(-1.5,4.0){$47$}
\put(-1.7,14){$46$}
\put(-21.7,24){$43$} \put(-1.7,24){$44$} \put(18.3,24){$45$}
\put(-11.7,34){$41$}\put(8.3,34){$42$}
\put(-1.7,44){$40$}
\put(-11.7,54){$38$}\put(8.3,54){$39$}
\put(-31.7,64){$34$}\put(-11.7,64){$35$}\put(8.3,64){$37$}\put(28.3,64){$36$}
\put(-11.7,74){$32$}\put(8.3,74){$33$}
\put(-21.7,84){$29$} \put(-1.7,84){$30$} \put(18.3,84){$31$}
\put(-11.7,94){$27$}\put(8.3,94){$28$}
\put(-31.7,104){$22$}\put(-11.7,104){$23$}\put(8.3,104){$25$}\put(28.3,104){$26$}\put(58.3,104){$24$}
\put(-11.7,114){$20$}\put(8.3,114){$21$}
\put(-21.7,124){$17$} \put(-1.7,124){$18$} \put(18.3,124){$19$}
\put(-11.7,134){$15$}\put(8.3,134){$16$}
\put(-31.7,144){$12$}\put(-11.7,144){$11$}\put(8.3,144){$13$}\put(28.3,144){$14$}
\put(-10.8,154){$9$}\put(8.3,154){$10$}
\put(-0.8,164){$8$}
\put(-10.8,174){$6$}\put(9.2,174){$7$}
\put(-20.8,184){$3$} \put(-0.8,184){$4$} \put(19.2,184){$5$}
\put(-0.8,194){$2$}\put(-0.8,204){$1$}

\end{picture}
\end{center}
\caption{$\Ext^1$ poset for $(E_8, 4, 4)$}
\label{jcfg4}
\end{figure}

\begin{remark}
The Jordan-H\"{o}lder length of $M_7$ is $83$. It is the longest basic generalized Verma module.

It is well-known that Hom spaces between Verma modules are at most one-dimensional. This is not the case for generalized Verma modules. The first counterexample was given in \cite{I1}. It is easy to see from the radical filtration of $M_{46}$ that
\[
\dim\Hom_{\caO^\frp}(M_{47}, M_{46})=2.
\]
Since the Jordan-H\"{o}lder length of $M_{46}$ is $3$, it is the shortest generalized Verma module with this property.
\end{remark}

\end{document}